\documentclass[11pt, a4paper, twoside,reqno]{amsart}
\usepackage[centering, totalwidth = 385pt, totalheight = 600pt]{geometry}

\usepackage{microtype}
\usepackage{xcolor}
\usepackage{amsmath}
\usepackage[all,cmtip,2cell]{xy}
\usepackage{amssymb}
\definecolor{darkGreen}{rgb}{0,0.45,0}
\usepackage[colorlinks,citecolor=darkGreen,linkcolor=darkGreen]{hyperref}
\usepackage{tikz}
\usepackage{scalefnt} 

\newtheorem{global-theorem}{Theorem}
\newtheorem{theorem}{Theorem}[section]
\newtheorem{lemma}[theorem]{Lemma}
\newtheorem{corollary}[theorem]{Corollary}
\newtheorem{proposition}[theorem]{Proposition}

\theoremstyle{definition}
\newtheorem{definition}[theorem]{Definition}
\newtheorem{example}[theorem]{Example}
\newtheorem{remark}[theorem]{Remark}
\newtheorem{conjecture}[theorem]{Conjecture}

\begin{document}

\title{Augmented Homotopical Algebraic Geometry}
\author{Scott Balchin}
\address{Department of Pure Mathematics\\
The Hicks Building\\
University of Sheffield\\
Sheffield S3 7RH, England, UK}
\email{scott.balchin@sheffield.ac.uk}

\begin{abstract}
We develop the framework for augmented homotopical algebraic geometry.  This is an extension of homotopical algebraic geometry, which itself is a homotopification of classical algebraic geometry.  To do so, we define the notion of augmentation categories, which are a special class of generalised Reedy categories.  For an augmentation category, we prove the existence of a closed Quillen model structure on the presheaf category which is compatible with the Kan-Quillen model structure on simplicial sets.  Moreover, we use the concept of augmented hypercovers to define a local model structure on the category of augmented presheaves.  We prove that crossed simplicial groups, and the planar rooted tree category are examples of augmentation categories. Finally, we introduce a method for generating new examples from old via a categorical pushout construction.
\end{abstract}

\maketitle

\section{Introduction}

Classically, algebraic geometry is the study of the zeros of multivariate polynomials, known as varieties.  Although this idea is still rooted in modern algebraic geometry, the tools used are now somewhat categorical and abstract.

The exodus towards categorical tools can, of course, be traced back to the school of Grothendieck in $20^\text{th}$ century France, with the seminal work of EGA and SGA~\cite{MR3075000,MR0476737}.  Of particular importance is the notion of a scheme, which subsumes the theory of affine varieties~\cite{MR0463157,MR1730819}.  For us, a scheme will be defined via its \emph{functor of points}, that is, a (1-)sheaf
\begin{equation}\label{defscheme} X \colon \textbf{Aff}^{op} \to \textbf{Set}, \end{equation}
where the topology on $\textbf{Aff}^{op}$ is usually taken to be the \'{e}tale topology~\cite{MR563524}.  This construction has led to a flourishing field of research fuelled by the emerging theory of categories, particularly that of $(\infty,1)$-categories.  We shall outline the natural modifications of the construction given by~(\ref{defscheme}) that have occurred throughout the years.  In turn, this will motivate the constructions that appear in this paper.

The first thing to note is how restrictive the category $\textbf{Set}$ is.  The study of moduli problems was marred by issues arising from non-trivial object isomorphisms (for example, isomorphisms of vector bundles over a scheme~\cite[\S 1.2]{MR2536079}).  To remedy this problem, the category of sets is enlarged to the (2-)category of groupoids~\cite{MR0220789,MR1771927}.  Therefore the definition of a \emph{stack} can be loosely worded as a (2-)sheaf  
\begin{equation}\label{defstack} X \colon \textbf{Aff}^{op} \to \textbf{Grpd}. \end{equation}

In some settings, the natural notion of equivalence for a geometric object is weaker than that of isomorphism.  To encode such information, a homotopical viewpoint is required.  The correct way to consider such objects is as an $(\infty,1)$-sheaf
\begin{equation}\label{defhighstack} X\colon \textbf{Aff}^{op} \to \widehat{\Delta}, \end{equation}
where $\widehat{\Delta}$ is the category of \emph{simplicial sets}.  Such objects are referred to as \emph{higher stacks}~\cite{hirsch,nstacks}.  A canonical example of a higher stack is that of the \emph{classifying stack} $K(G,n)$ for $G$ an abelian $k$-group scheme~\cite{whatis}.

The final stage of the current progression of construction~(\ref{defscheme}), unsurprisingly, is an enlargement of the source category of affine schemes.  Just as the category of simplicial sets is a homotopic version of \textbf{Set}, we can consider a homotopic version of affine schemes.  The resulting category is the $(\infty,1)$-category of \emph{derived affine schemes}.  A \emph{derived stack} is then an $(\infty,1)$-sheaf
\begin{equation}\label{defderstack} X\colon \textbf{dAff}^{op} \to \widehat{\Delta}. \end{equation}

The scope of this paper is to develop the tools necessary to take the above progression one step further by enlarging the simplex category $\Delta$ to an \emph{augmentation category} $\mathbb{A}$.  By considering the category of presheaves over an augmentation category, we get an extension of the category of simplicial sets. Via the construction of specific Quillen model structures, we will be able to encode the correct ``sheaf condition''.  Just as each modification of construction~(\ref{defscheme}) was undertaken to accommodate specific issues, the correct choice of an augmentation category can be used to encode necessary data.  Consider the situation where there is a need to encode an $SO(2)$-action on a derived stack.  Such an action cannot be inherently captured by the category of simplicial sets.  However, this problem is perfectly suited to the category of cyclic sets $\widehat{{\Delta \mathfrak{C}}}$, which is formed as a presheaf category on Connes' cyclic category ${\Delta \mathfrak{C}}$~\cite{connes1}.   Using the general framework of augmentation categories, we will show that it is possible to adjust construction~(\ref{defscheme}) to define \emph{$SO(2)$-equivariant derived stacks} as certain functors
\begin{equation}\label{defderstack22} X\colon \textbf{dAff}^{op} \to \widehat{{\Delta \mathfrak{C}}}. \end{equation}

All of the above discussion can be assembled as in Figure~\ref{diagramrel}, which is an extension of the one appearing in~\cite{whatis}:

\begin{figure}[ht]
  \centering
  \begin{minipage}[b]{0.4\textwidth}
  \[\begin{gathered}
  	\xymatrixcolsep{6pc}\xymatrixrowsep{4pc}\xymatrix{\textbf{Aff}^{op} \ar[dr]|{\mathcal{B}} \ar[ddr]|{\mathcal{C}} \ar[dd] \ar[r]|{\mathcal{A}} & \textbf{Set} \ar[d] \\
& \textbf{Grpd} \ar[d] \\
\textbf{dAff}^{op} \ar[r]|{\mathcal{D}} \ar[dr]|{\mathcal{E}}& \widehat{\Delta} \ar[d] \\ & \widehat{\mathbb{A}}}
\end{gathered}\]
  \end{minipage}
  \hfill
  \begin{minipage}[b]{0.5\textwidth}
	\begin{tabular}{l}
$\mathcal{A}$ - Schemes \\[10pt]
$\mathcal{B}$ - Stacks  \\[10pt]
$\mathcal{C}$ - Higher Stacks \\[10pt]
$\mathcal{D}$ - Derived Stacks \\[10pt]
$\mathcal{E}$ - Augmented Derived Stacks        
\end{tabular}
\end{minipage}
\caption{The diagram of relations for augmented derived algebraic geometry}\label{diagramrel}
\end{figure}
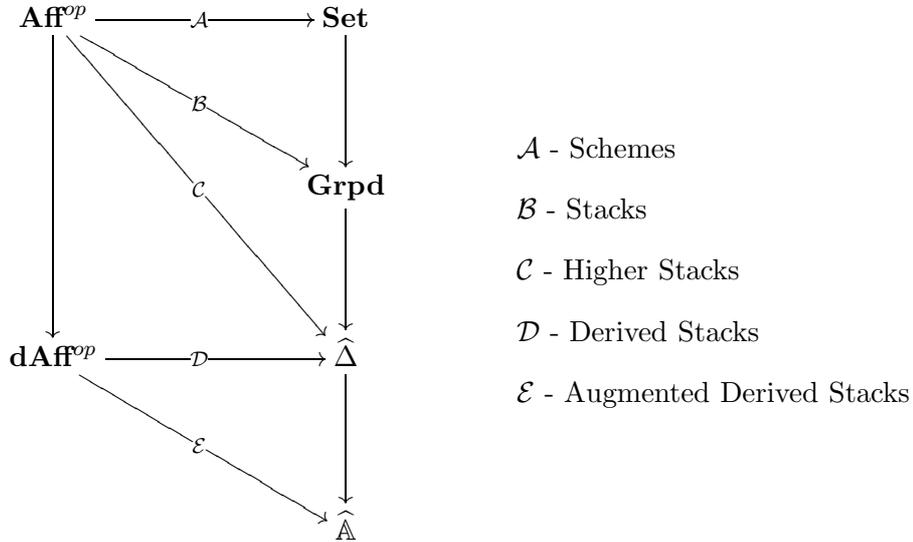

So far we have only considered the \emph{algebraic} setting, in that we work over the algebraic category of (derived) affine schemes.  We shall see that we can modify our source category to any (simplicial) site.  Classically, of particular significance are the categories \textbf{Top} (resp., \textbf{Mfd}) which give rise to topological (resp., differentiable) stacks~\cite{MR2817778,MR2977576,noohifound}.  

\subsection*{Outline of Paper}

\begin{itemize}
\item In Section~\ref{sec:2} we recall the theory of homotopical algebraic geometry via the construction of certain closed Quillen model structures.  Moreover, we recall the definition of the category of derived affine schemes, and give the construction of $n$-geometric derived stacks via $n$-hypergroupoids.
\item In Section~\ref{sec:3} we give the required machinery for the main definition of the paper, that of an augmentation category (Definition~\ref{augdef}).
\item In Section~\ref{sec:4} the homotopy theory of augmentation categories is discussed.  The existence of a closed Quillen model structure is proven for any augmentation category.  The model category relies on the construction of \emph{augmented Kan complexes} which mirror the simplicial analogue.  
\item In Section~\ref{sec:5} we develop the theory of augmented homotopical algebraic geometry via a Quillen model structure on the category of augmented presheaves.  Moreover, we define $n$-geometric augmented derived stacks using a modified $n$-hypergroupoid construction.
\item In Section~\ref{sec:6} we prove that crossed simplicial groups provide a class of examples of augmentation categories, and that the resulting geometric theory can be thought of as being equivariant.  We briefly discuss two applications of equivariant stacks which the author has developed in \cite{balchin,balchin2}.
\item In Section~\ref{sec:7} we prove that the category of planar rooted trees appearing in the theory of dendroidal sets is an augmentation category, with the resulting geometric theory being stable.  We then discuss an amalgamation property for augmentation categories which allows us to form new examples from old.
\end{itemize}

\section{Homotopical Algebraic Geometry}\label{sec:2}

We shall now utilise the theory of simplicial objects and Quillen model structures to discuss the theory of homotopical algebraic geometry.  A nice overview of the constructions in this section can be found in~\cite{MR3285853}.  We begin by recalling the notion of a \emph{site}~\cite{artin1962grothendieck}.

\begin{definition}
Let $\mathcal{C}$ be a category with all fiber products.  A \emph{Grothendieck topology} on $\mathcal{C}$ is given by a function $\tau$ which assigns to every object $U$ of $\mathcal{C}$ a collection $\tau(U)$, consisting of families of morphisms $\{\varphi_i \colon U_i \to U\}_{i \in I}$ with target $U$ such that the following axioms hold:

\begin{enumerate}
\item (Isomorphisms) If $U' \to U$ is an isomorphism, then $\{U' \to U\}$ is in $\tau(U)$.
\item (Transitivity) If the family $\{\varphi_i \colon U_i \to U\}$ in $\tau(U)$ and if for each $i \in I$ one has a family $\{ \varphi_{ij} : U_{ij} \to U_i\}_{j \in J}$ in $\tau(U_i)$ then the family $\{\varphi_{i}\circ \varphi_{ij} \colon U_{ij} \to U\}_{i \in I, j \in J}$ is in $\tau(U)$.
\item (Base change) If the family $\{\varphi_i \colon U_i \to U\}_{i \in I}$ is in $\tau(U)$ and if $V \to U$ is any morphism, then the family $\{V \times_U U_i \to V\}$ is in $\tau(V)$.
\end{enumerate}

The families in $\tau(U)$ are called \emph{covering families} for $U$ in the $\tau$-topology.  We call a category $\mathcal{C}$ with such a topology a \emph{site}, denoting it as $(\mathcal{C},\tau)$.
\end{definition}

Note that we can consider the case where our sites are simplicially enriched, such that the topology is defined on $\pi_0\mathcal{C}$.  We shall use this assumption throughout without ever making it explicit.  

The first task is to now construct the category of $\infty$-stacks on a site $\mathcal{C}$ using model structures in the sense of~\cite{hovey1999model}.  For a category $\mathcal{C}$, write $\widehat{\mathcal{C}}$ for the category of presheaves on $\mathcal{C}$.  We will denote by $\widehat{\Delta}_{\text{Kan}}$ the category of simplicial sets equipped with the Kan-Quillen model structure.  Let $\mathcal{C}$ be a site, the category of \emph{simplicial presheaves} on $\mathcal{C}$ is the functor category $\textbf{sPr}(\mathcal{C}) := \widehat{\Delta}^{\mathcal{C}^{op}}$.  As $\widehat{\Delta}_\text{Kan}$ is left-roper and combinatorial, we can equip $\textbf{sPr}(\mathcal{C})$ with the projective model structure, $\textbf{sPr}_\text{proj}(\mathcal{C})$, which is also left-proper and combinatorial.  The problem with this first model structure is that it does not see any of the topology $\tau$.  To encode this extra data, we add more weak equivalences, in effect, enforcing an $\infty$-sheaf condition.

\begin{definition}
Let $(\mathcal{C}, \tau)$ be a site.  A map $f \colon \mathcal{F} \to \mathcal{F}'$ in $\textbf{sPr}(\mathcal{C})$ is a \textit{local weak equivalence} if:
\begin{itemize}
\item The induced map $\pi_0^\tau \mathcal{F} \to \pi_0^\tau\mathcal{F}'$ is an isomorphism of sheaves, where $\pi_0^\tau$ is the sheafification of $\pi_0$.
\item Squares of the following form are pullbacks after sheafification:
$$
\xymatrix{\pi_n \mathcal{F} \ar[r] \ar[d] &\ar[d] \pi_n \mathcal{F}' \\
\mathcal{F}_0 \ar[r] &  \mathcal{F}'_0 \rlap{ .}}
$$
\end{itemize}
\end{definition}

\begin{theorem}[{\cite[\S 3]{jardinesimplicialpresheaves}}]\label{projlocmodel}
Let $(\mathcal{C}, \tau)$ be a site.  There exists a cofibrantly generated \textit{local model structure} on the category $\textbf{sPr}(\mathcal{C})$ where a map $f \colon \mathcal{F} \to \mathcal{F}'$ is a:
\begin{itemize}
\item Weak equivalence if it is a local weak equivalence.
\item Fibration if it has the RLP with respect to the trivial cofibrations.
\item Cofibration if it is a cofibration in $\textbf{sPr}_\text{proj}(\mathcal{C})$.
\end{itemize}
We will denote this model structure $\textbf{sPr}_\tau(\mathcal{C})$.
\end{theorem}

Before continuing, we introduce a different description of the local model structure using left Bousfield localisation (as already noted, $\textbf{sPr}_\text{proj}(\mathcal{C})$ is left-proper and combinatorial so such a localisation exists). 

\begin{definition}
Let $(\mathcal{C},\tau)$ be a site.  A morphism $f \colon \mathcal{F} \to \mathcal{F}'$ in $\textup{Ho}(\textbf{sPr}_\text{proj}(\mathcal{C}))$ is called a \emph{$\tau$-covering}  if the induced map $\pi_0^\tau (\mathcal{F}) \to \pi_0^\tau (\mathcal{F}')$ is an epimorphism of sheaves.
\end{definition}

\begin{definition}\label{hyperdef}
Let $(\mathcal{C}, \tau)$ be a site.  A map $f \colon X \to Y$ in $\textbf{sPr}(\mathcal{C})$ is a \textit{hypercovering} if for all $n \in \mathbb{N}$:
$$X_n \to \text{Hom}_{\widehat{\Delta}}(\partial \Delta[n] , X) \times_{\text{Hom}_{\widehat{\Delta}}(\partial \Delta[n] , Y)} Y_n$$
is a $\tau$-covering in $\textup{Ho}(\textbf{sPr}_\text{proj}(\mathcal{C}))$.
\end{definition}

\begin{theorem}[{\cite[Theorem 6.2]{MR2034012}}]
There is a model structure on $\textbf{sPr}(\mathcal{C})$ which is the left Bousfield localisation of $\textbf{sPr}_\text{proj}(\mathcal{C})$ at the class of hypercovers.  Moreover, this model is Quillen equivalent to $\textbf{sPr}_\tau(\mathcal{C})$.
\end{theorem}

\begin{definition}
For a site $(\mathcal{C}, \tau)$ the homotopy category $\textup{Ho}(\textbf{sPr}_\tau(\mathcal{C}))$ will be referred to as the category of \emph{$\infty$-stacks on $\mathcal{C}$}. 
\end{definition}

\subsection{Closed Monoidal Structure}

We now prove the existence of a closed monoidal structure on $\textup{Ho}(\textbf{sPr}_\tau(\mathcal{C}))$, where the monoidal structure is given by direct product. To do so, we show that $\textbf{sPr}_\tau(\mathcal{C})$ is a monoidal model category.  

\begin{definition}\label{monmodcat}
A (symmetric) monoidal model category is a model category $\mathcal{C}$ equipped with a closed (symmetric) monoidal structure $(\mathcal{C}, \otimes, I)$ such that the two following compatibility conditions are satisfied:
\begin{enumerate}
\item (Pushout-product axiom) For every pair of cofibrations $f \colon X \to Y$ and $f' \colon X' \to Y'$, their pushout-product
$$(X \otimes Y') \bigsqcup_{X \otimes X'} (Y \otimes X')\to Y \otimes Y'$$
is also a cofibration.  Moreover, it is a trivial cofibration if either $f$ or $f'$ is.
\item (Unit axiom) For every cofibrant object $X$ and every cofibrant resolution $0 \hookrightarrow QI \to I$ of the tensor unit, the induced morphism $QI \otimes X \to I \otimes X$ is a weak equivalence.  (Note that this holds automatically in the case that $I$ is cofibrant).
\end{enumerate}
\end{definition}

The proof of this property for $\textup{Ho}(\textbf{sPr}_\tau(\mathcal{C}))$ is non-trivial unless you make a small adjustment to the model that you are considering.  The problem that we currently face is that the class of cofibrations in $\textbf{sPr}_\tau(\mathcal{C})$ are not easily described. This is because in the projective model they are simply defined via a lifting property.   Instead, we choose a model which is Quillen equivalent to the one that we are interested in, and prove the existence of the monoidal model category structure for this.  The following proposition is proved in the same manner as Theorem~\ref{projlocmodel}.

\begin{proposition}
There is a model structure on $\textbf{sPr}(\mathcal{C})$ called the \emph{injective local model} where a map $f \colon \mathcal{F} \to \mathcal{F}'$ is a:
\begin{itemize}
\item Weak equivalence if it is a local weak equivalence.
\item Fibration if it has the RLP with respect to the trivial cofibrations.
\item Cofibration if it is a cofibration in the injective model on $\textbf{sPr}(\mathcal{C})$ (i.e., a point-wise monomorphism).
\end{itemize}
We shall denote by $\textbf{sPr}_{\textup{inj},\tau}(\mathcal{C})$ this model structure.
\end{proposition}

\begin{lemma}
$\textbf{sPr}_{\textup{inj},\tau}(\mathcal{C})$ is a closed monoidal model category.
\end{lemma}

\begin{proof}
We need to show that the pushout-product and unit axiom of Definition~\ref{monmodcat} hold.  However, as the cofibrations are monomorphisms, this is trivial.
\end{proof}

\begin{corollary}
$\textup{Ho}(\textbf{sPr}_{\textup{inj},\tau}(\mathcal{C})) \simeq \textup{Ho}(\textbf{sPr}_{\tau}(\mathcal{C}))$ is a closed monoidal category.  The internal-hom in $\textup{Ho}(\textbf{sPr}_{\tau}(\mathcal{C}))$ will be denoted $\mathbb{R}\underline{\mathcal{H}\text{om}}(-,-)$.
\end{corollary}

We will refer to $\mathbb{R}\underline{\mathcal{H}\text{om}}(X,Y)$ as the \emph{derived mapping stack} from $X$ to $Y$ for two $\infty$-stacks $X,Y$.  One can compute this for each object $c \in \mathcal{C}$:
$$\mathbb{R}\underline{\mathcal{H}\text{om}}(F,G)(c) := \underline{\text{hom}}^\Delta(c \times F,R_\textup{inj}G),$$
where $R_\textup{inj}$ is the fibrant replacement functor in $\textbf{sPr}_\textup{inj}(\mathcal{C})$ and $\underline{\text{hom}}^\Delta$ is the simplicial-hom of $\text{sPr}(\mathcal{C})$.  We will denote by $\mathbb{R}\underline{\text{Hom}}$(-,-) the derived simplicial-hom.

\subsection{Geometric Homotopical Stacks}

We begin with a warning. There is a slight discrepancy between definitions in the literature.  The two main bodies of work in the literature are that of To\"{e}n-Vezzosi~\cite{MR2394633} and Lurie~\cite{MR2717174}, in which the notions of $n$-geometric stacks are slightly different.  From now on we will follow the conventions of~\cite{MR2394633}.

In the study of algebraic stacks, one is not interested in all objects.  Instead, one is interested in only those stacks which have a smooth or \'{e}tale atlas (this can also be worded in terms of internal groupoid objects).  We say that stacks with such an atlas are \emph{geometric}.  In this section we will introduce what we mean by a geometric $\infty$-stack.  We do this in two ways: the first uses an atlas representation as in the classical setting.  This method, however, is cumbersome for practical purposes.  Therefore we introduce the second method which uses the theory of \emph{hypergroupoids}.

We shall be interested in the site of \emph{derived affine schemes}, which we will denote \textbf{dAff}.  An introduction to this site can be found in~\cite{akhil}, while details of the \'{e}tale topology appear in~\cite[Chapter 2.2]{MR2394633}.

\subsubsection{Via Iterated Representability}\label{iterrep}

We fix our site $\mathcal{C} = \textbf{dAff}$, and a set of covering maps $\textbf{P}$ (usually we take $\textbf{P}$ to mean smooth or \'{e}tale).

\begin{definition}[{\cite[Definition 1.3.3.1]{MR2394633}}]
\leavevmode
\begin{enumerate}
\item A stack is \emph{$(-1)$-geometric} if it is representable.
\item A morphism of stacks $f \colon F \to G$ is \emph{$(-1)$-representable} if for any representable stack $X$ and any morphism $X \to G$ the homotopy pullback $F \times ^h_G X$ is a representable stack.
\item A morphism of stacks  $f \colon F \to G$ is \emph{in $(-1)$-\textbf{P}} if it is $(-1)$-representable and if for any representable stack $X$ and any morphism $X \to G$, the induced morphism $F \times_G^h X \to X$ is a $\textbf{P}$-morphism between representable stacks.
\end{enumerate}
We now let $n \geq 0$.
\begin{enumerate}
\item Let $F$ be any stack.  An \emph{$n$-atlas} for $F$ is a small family of morphisms $\{U_i \to F\}_{i \in I}$ such that
\begin{enumerate}
	\item Each $U_i$ is representable.
	\item Each morphism $U_i \to F$ is in $(n-1)$-$\textbf{P}$.
	\item The total morphism $\coprod_{i \in I} U_i \to F$ is an epimorphism.
\end{enumerate}
\item A stack $F$ is \emph{$n$-geometric} if it satisfies the following conditions:
\begin{enumerate}
	\item The diagonal $F \to F \times^h F$ is $(n-1)$-representable.
	\item The stack $F$ admits an $n$-atlas.
\end{enumerate}
\item A morphism of stacks $F \to G$ is \emph{$n$-representable} if for any representable stack $X$ and any morphism $X \to G$ the homotopy pullback $F \times^h_G X$ is $n$-geometric.
\item A morphism of stacks $F \to G$ is \emph{in $n$-$\textbf{P}$} if it is $n$-representable and if for any representable stack $X$, any morphism $X \to G$, there exists an $n$-atlas $\{U_i\}$ of $F \times^h_G X$ such that each composite morphism $U_i \to X$ is in $\textbf{P}$.
\end{enumerate}
We will say that a stack is \emph{geometric} if it is $n$-geometric for any $n$.
\end{definition}

\begin{definition}\label{index:gest}
The  full subcategory of $n$-geometric stacks of $\textup{Ho}(\textbf{sPr}_\tau(\mathcal{C}))$ will be denoted $\textbf{GeSt}_n(\mathcal{C},\textbf{P})$. In particular:
\begin{itemize}
\item $\textbf{GeSt}_n(\textbf{dAff},\text{sm})$ is the category of \emph{derived $n$-Artin stacks}.
\item $\textbf{GeSt}_n(\textbf{dAff},\text{\'{e}t})$ is the category of \emph{derived $n$-Deligne-Mumford stacks}.
\end{itemize}
\end{definition}

\subsubsection{Via $n$-Hypergroupoids}

In this section, we introduce a second representation of $n$-geometric stacks.  This method, via hypergroupoids, is more intuitive than the method presented in Section~\ref{iterrep}, and is reportedly closer to that envisaged by Grothendieck in~\cite{pursuing}.  An easy to read overview of the theory can be found in the paper~\cite{pridham2}, while the full results in the most general setting can be found in~\cite{MR3033634}.  The idea is that an $n$-geometric stack on the site $\mathcal{C}$ can be resolved by some fibrant object in the category $\mathcal{C}^{\Delta^{op}}$.  The impressive thing about this construction is that it completely avoids the need for the local model structure on simplicial presheaves, but is intricately linked to the notion of hypercovers nonetheless.  Note that an $n$-hypergroupoid captures the nerve construction of an $n$-groupoid~\cite{MR555549}.

\begin{definition}\label{hypergrp}
An $n$-hypergroupoid is an object $X \in {\widehat{\Delta}}$ for which the maps
$$X_m = \text{Hom}_{\widehat{\Delta}}(\Delta[m],X) \to \text{Hom}_{\widehat{\Delta}}(\Lambda^k[m],X)$$
are surjective for all $m$, $k$, and isomorphisms for all $m >n$ and all $k$.  In particular, $X$ is a Kan complex, and therefore fibrant in ${\widehat{\Delta}}_\text{Kan}$.
\end{definition}

It is possible to characterise $n$-hypergroupoids using the coskeleton construction. Recall that the \emph{$m$-coskeleton}, denoted $\text{cosk}_mX$, is defined to be $(\text{cosk}_mX)_i = \text{Hom}(\Delta[i]_{\leq m}, X_{\leq m})$ where $X_{\leq m}$ is the truncation at $m$. 

\begin{lemma}[{\cite[Proposition 3.4]{MR2112899}}]
An $n$-hypergroupoid $X$ is completely determined by its truncation $X_{\leq n+1}$. In fact, $X = \text{cosk}_{n+1}X$ (i.e., $X$ is $(n+1)$-coskeletal).  Conversely, a simplicial set of the form $\text{cosk}_{n+1}X$ is an $n$-hypergroupoid if and only if it satisfies the conditions of Definition~\ref{hypergrp} up to level $n+2$.
\end{lemma}

\begin{definition}
A morphism $f \colon X \to Y$ in ${\widehat{\Delta}}$ is a \emph{trivial relative $n$-hypergroupoid} if the  maps
$$X_m \to \text{Hom}_{\widehat{\Delta}}(\partial \Delta[m] , X) \times_{\text{Hom}_{\widehat{\Delta}}(\partial \Delta[m] , Y)} Y_m$$
are surjections for all $m$, and isomorphisms for all $m \geq n$.  In particular, $f$ is a trivial fibration in ${\widehat{\Delta}}_\text{Kan}$.
\end{definition}

\begin{lemma}[{\cite[Lemma 2.9]{MR3033634}}]\label{cosktriv}
Let $f \colon X \to Y$ be a trivial $n$-hypergroupoid. Then $X = Y \times_{\text{cosk}_{n-1} Y} \text{cosk}_{n-1}X$.
\end{lemma}

We can compare the above definition and property of trivial $n$-hypergroupoids with the definition of a hypercover (Definition~\ref{hyperdef}), and observe that a trivial $n$-hypergroupoid can be seen as a \emph{truncated} or \emph{bounded} hypercover.

We can move to a geometric setting by considering objects in $\textbf{sdAff} := \textbf{dAff}^{\Delta^{op}}$.  We also change the surjectivity condition to be surjectivity with respect to a class of covering maps (i.e., smooth or \'{e}tale). 

\begin{definition}
A \emph{derived Artin (resp., Deligne-Mumford) $n$-hypergroupoid} is a simplicial derived affine scheme $X \in \textbf{sdAff}$ such that the  maps
$$X_m = \text{Hom}_{\widehat{\Delta}}(\Delta[m],X) \to \text{Hom}_{\widehat{\Delta}}(\Lambda^k[m],X)$$
are smooth (resp., \'{e}tale) surjections for all $m$, $k$, and isomorphisms for all $m >n$ and all $k$. 
\end{definition}

Note that given an (Artin or Deligne-Mumford) $n$-hypergroupoid $X$, we can construct a simplicial presheaf $X$ on $\textbf{dAff}$ as follows:
\begin{align*}
X \colon \textbf{dAff}^{op} &\to {\widehat{\Delta}},\\
X(A)_n & := X_n(A).
\end{align*}

\begin{definition}
A morphism $f \colon X \to Y$ in $\textbf{sdAff}$ is a \emph{trivial relative derived Artin (resp., Deligne-Mumford) $n$-hypergroupoid} if the  maps
$$X_m \to \text{Hom}_{\widehat{\Delta}}(\partial \Delta[m] , X) \times_{\text{Hom}_{\widehat{\Delta}}(\partial \Delta[m] , Y)} Y_m$$
are smooth (resp., \'{e}tale) surjections for all $m$, and isomorphisms for all $m \geq n$ (i.e., are $n$-truncated hypercovers).
\end{definition}

Recall from~\cite{MR2877401} that a \emph{relative category} is a pair $(\mathcal{C},\mathcal{W})$ consisting of a category $\mathcal{C}$ and a wide subcategory $\mathcal{W}$ whose maps are called \emph{weak equivalences}.  Such a category has a homotopy category by formally inverting all of the weak equivalences.

\begin{definition}\label{index:hyper}
The \emph{category of derived $n$-Artin (resp., Deligne-Mumford) stacks} is obtained from the full subcategory of $\textbf{sdAff}$ consisting of the relative category of derived Artin (resp., Deligne-Mumford) $n$-hypergroupoids and the trivial relative derived Artin (resp., Deligne-Mumford) $n$-hypergroupoids as the weak equivalences.  We will denote this category $\mathcal{G}_n^\text{sm}(\textbf{dAff})$ (resp., $\mathcal{G}_n^\text{\'{e}t}(\textbf{dAff})$).
\end{definition}

The following theorem is the main result from~\cite{MR3033634}, and proves that we can move freely (up to homotopy) between the $n$-hypergroupoid and $n$-geometric stack constructions.

\begin{theorem}[{\cite[Theorem 4.15]{MR3033634}}]\label{equivcat}
There is an equivalence of categories
\begin{align*}
\textbf{GeSt}_n(\textbf{dAff},\text{sm})  & \simeq \mathcal{G}_n^\text{sm}(\textbf{dAff}), \\
\textbf{GeSt}_n(\textbf{dAff},\text{\'{e}t}) & \simeq \mathcal{G}_n^\text{\'{e}t}(\textbf{dAff}).
\end{align*}
In fact, such an equivalence can be formulated for any homotopical algebraic geometric context \cite[Definition 1.3.2.13]{MR2394633}.
\end{theorem}

\section{Augmentation Categories}\label{sec:3}

In this section we recall the theory of generalised Reedy categories from~\cite{genreedy}, and use them to formulate the framework of an \emph{augmentation category}.

\subsection{Generalised Reedy Categories}

A (strict) Reedy category $\mathbb{S}$ is a category such that we can equip $\mathcal{E}^{\mathbb{S}^{op}}$ with a model structure, for $\mathcal{E}$ a cofibrantly generated model category \cite{strictreedy}.  The classes of maps in this model structure can be described explicitly using those of $\mathcal{E}$.  An example of a strict Reedy category is the simplex category $\Delta$.  One shortcoming of strict Reedy categories is that they do not allow for non-trivial automorphisms on the objects, which occur, for example, in the cyclic category of Connes, introduced in \cite{connes1}.  A \emph{generalised Reedy category} allows us to capture this automorphism data.  Recall that a subcategory $\mathcal{D} \subset \mathcal{C}$ is said to be \emph{wide in $\mathcal{C}$} if $\text{Ob}(\mathcal{C}) = \text{Ob}(\mathcal{D})$. The following definition appears in \cite{genreedy}.

\begin{definition}[{\cite[Definition 1.1]{genreedy}}]\label{def:genreedy}
A \textit{generalised Reedy structure} on a small category $\mathbb{R}$ consists of wide subcategories $\mathbb{R}^+$, $\mathbb{R}^-$, and a degree function $d \colon \text{Ob}(\mathbb{R})\to \mathbb{N}$ satisfying the following four axioms:
\begin{enumerate}
\item[i)] Non-invertible morphisms in $\mathbb{R}^+$ (resp., $\mathbb{R}^-$) raise (resp., lower) the degree; isomorphisms in $\mathbb{R}$ preserve the degree.
\item[ii)] $\mathbb{R}^+ \cap \mathbb{R}^- = \text{Iso}(\mathbb{R})$.
\item[iii)] Every morphism $f$ of $\mathbb{R}$ factors uniquely (up to isomorphism) as $f=gh$ with $g \in \mathbb{R}^+$ and $h \in \mathbb{R}^-$.
\item[iv)] If $\theta f = f$ for $\theta \in \text{Iso}(\mathbb{R})$ and $f \in \mathbb{R}^-$, then $\theta$ is an identity.
Moreover, we say that the generalised Reedy structure is \emph{dualisable}  if the following additional axiom holds:
\item[iv)$'$] If $f \theta = f$ for $\theta \in \text{Iso}(\mathbb{R})$ and $f \in \mathbb{R}^+$, then $\theta$ is an identity.
\end{enumerate}
A \emph{morphism} of generalised Reedy categories is a functor $\mathbb{R} \to \mathbb{R}'$ which takes $\mathbb{R}^+$ (resp., $\mathbb{R}^-$) to $\mathbb{R}'^+$ (resp., $\mathbb{R}'^-$) and preserves the degree.
\end{definition}

It is possible to generate a large class of generalised Reedy categories using the theory of crossed groups on categories \cite{MR923136}.  

\begin{definition}[{\cite[Proposition 2.5]{genreedy}}]\label{csgrpdeff}
Let $\mathbb{R},\mathbb{S}$ be categories such that $\mathbb{R} \subseteq \mathbb{S}$ is a wide subcategory. Assume that for all $s \in \mathbb{S}$, there exist subgroups $\mathfrak{G}_s \subseteq \text{Aut}_\mathbb{S}(s)$ of special automorphisms such that each morphism in $\mathbb{S}$ factors uniquely as a special automorphism followed by a morphism in $\mathbb{R}$.  Then $\mathbb{S}$ is a \emph{crossed $\mathbb{R}$-group}, which we denote $\mathbb{R}\mathfrak{G}$.  What we call a crossed $\mathbb{R}$-group is sometimes referred to as the \emph{total category} of a crossed $\mathbb{R}$-group.
\end{definition}

There is a compatibility condition appearing in the following proposition which we will not cover as all categories that we will consider will satisfy it \cite[Remark 2.9]{genreedy}.

\begin{proposition}[{\cite[Proposition 2.10]{genreedy}}]\label{genreedycrossed}
Let $\mathbb{R}$ be a strict Reedy category, and $\mathbb{R} \mathfrak{G}$ a compatible crossed $\mathbb{R}$ group.  Then there is a unique dualisable generalised Reedy structure on $\mathbb{R} \mathfrak{G}$ for which the embedding $\mathbb{R} \hookrightarrow \mathbb{R} \mathfrak{G}$ is a morphism of generalised Reedy categories.
\end{proposition}

We shall now use the degree function appearing in the Definition \ref{def:genreedy} to define the notion of (co)skeleton.  Denote by $\mathbb{R}_{\leq n}$ the subcategory of $\mathbb{R}$ consisting of objects of degree $\leq n$.  Write $t_n \colon \mathbb{R}_{\leq n} \hookrightarrow \mathbb{R}$ for the corresponding full embedding.

\begin{definition}[{\cite[Definition 6.1]{genreedy}}]
Let $\mathbb{R}$ be a generalised Reedy category.
\begin{itemize}
\item The \emph{$n$-skeleton} functor is the endofunctor $\textup{sk}_n := t_{n!}t^\ast_n$.
\item The \emph{$n$-coskeleton} functor is the endofunctor $\textup{cosk}_n := t_{n\ast}t^\ast_n$.
\end{itemize}
\end{definition}

The class of EZ-categories are a subclass of generalised Reedy categories for which the skeletal filtrations have a nice description.  These skeletal filtrations can in turn be described by a corresponding \emph{boundary object} using notions of \emph{face} and \emph{degeneracy} maps analogous to the simplicial case.  These boundary objects will allow us to give an explicit description of when objects in the presheaf category are coskeletal.

\begin{definition}[{\cite[Definition 6.6]{genreedy}}]
An \emph{EZ-category} (Eilenberg-Zilber category) is a small category $\mathbb{R}$, equipped with a degree function $d \colon \text{Ob}(\mathbb{R}) \to \mathbb{N}$, such that
\begin{enumerate}
\item Monomorphisms preserve (resp., raise) the degree if and only if they are invertible (resp., non-invertible).
\item Every morphism factors as a split epimorphism followed by a monomorphism.
\item Any pair of split epimorphisms with common domain gives rise to an absolute pushout (recall an \emph{absolute pushout} is a pushout preserved by the Yoneda embedding $\mathbb{R} \hookrightarrow \widehat{\mathbb{R}}$).
\end{enumerate}
An EZ-category is a dualisable generalised Reedy category with $\mathbb{R}^+$ (resp., $\mathbb{R}^-$) defined to be the wide subcategory containing all monomorphisms (resp., epimorphisms).  Moreover, we will say that an EZ-category $\mathbb{R}$ is \emph{symmetric promagmoidal} if $\widehat{\mathbb{R}}$ has a symmetric tensor product $(\widehat{\mathbb{R}},\square,I_\square)$.  Clearly any presheaf category carries the cartesian product, but often the tensor structure that we work with will be different to the cartesian product.
\end{definition}

\begin{definition}
Let $\mathbb{R}$ be an EZ-category.  Denote by $\mathbb{R}[r]$ the representable presheaf of $r \in \mathbb{R}$ in the topos $\widehat{\mathbb{R}}$.  The split-epimorphisms will be called the \emph{degeneracy operators} and the monomorphisms will be called the \emph{face operators}.  
\end{definition}

\begin{definition}
Let $\mathbb{R}$ be an EZ-category and $r \in \mathbb{R}$.  The \emph{boundary}, $\partial \mathbb{R}[r] \subset \mathbb{R}[r]$ is the subobject of those elements of $\mathbb{R}[r]$ which factor through a non-invertible face operator $s \to r$. Explicitly, 
$$\partial \mathbb{R}[r] = \bigcup_{f \colon s \to r} f(\mathbb{R}[s]).$$
\end{definition}

\begin{lemma}[{\cite[Corollary 6.8]{genreedy}}]\label{isskele}
Let  $\mathbb{R}$ be an EZ-category and $r \in \mathbb{R}$, then $\partial \mathbb{R}[r]= \textup{sk}_{d(r)-1}\mathbb{R}[r]$. 
\end{lemma}

The above definition of boundary coincides exactly with the definition of boundary in the simplicial case.  We can now say when an object $X \in \widehat{\mathbb{R}}$ is coskeletal.

\begin{lemma}\label{coskrep}
Let $\mathbb{R}$ be an EZ-category, and $X \in \widehat{\mathbb{R}}$.  Then the following are equivalent:
\begin{enumerate}
\item The unit of the adjunction $X \to \textup{cosk}_n(X)$ is an isomorphism.
\item The map $X_r =\textup{Hom}(\mathbb{R}[r],X) \to \textup{Hom}(\mathbb{R}[r]_{\leq n},X_{\leq n})$ is a bijection for all $r$ with $d(r) >n$.
\item For all $r$ with $d(r) > n$, and every morphism $\partial  \mathbb{R}[r] \to X$, there exists a unique filler $\mathbb{R}[r] \to X$:
$$\xymatrix{\partial \mathbb{R}[r] \ar[r] \ar[d]& X\ \\
\mathbb{R}[r] \ar@{-->}[ur]}$$
\end{enumerate}
If $X$ satisfies any of these equivalent definitions we shall say that $X$ is \emph{$n$-coskeletal}.
\end{lemma}

\begin{proof}
The equivalence of the first two conditions follows from the definition.  For the final condition, note that $\textup{cosk}_k(X)$ if given by the formula $[n] \mapsto \text{Hom}(\textup{sk}_k(\Delta[n]) , X)$ by adjointness, then using Lemma \ref{isskele} we see that this is the same as $\text{Hom}(\partial \mathbb{R}[r],X)$ for $d(r) > n $.  The unique filler condition is then equivalent to Condition 2.
\end{proof}

\subsection{Augmentation Categories}

\begin{definition}[{\cite[Proposition 7.2]{genreedy}}]\label{thecofibrations}
Let $\mathbb{R}$ be an EZ-category.  A \emph{normal monomorphism} in $\widehat{\mathbb{R}}$ is a map $f\colon X \to Y$ such that $f$ is monic and for each object $r$ of $\mathbb{R}$ and each non-degenerate element $y \in Y_r \backslash f(X)_r$, the isotropy group $\{ g \in \text{Aut} (r) \mid g^\ast(y)= y \}$ is trivial. 
\end{definition}

We now have all of the necessary tools to introduce what we mean by an augmentation category.  In the following sections we will be exploring the properties of these categories alongside developing a geometric framework for them.

\begin{definition}\label{augdef}
An \emph{augmentation category} is a category $\mathbb{A}$ such that:
\begin{itemize}
\item[(AC1)] $\mathbb{A}$ is a symmetric promagmoidal EZ-category. 
\item[(AC2)] There is a faithful inclusion of EZ-categories $i \colon \Delta \hookrightarrow \mathbb{A}$ such that for any two simplicial sets $X$ and $Y$ we have $i_! (X) \square i_!(Y) \simeq i_!(X \times Y)$, where $\square$ is the tensor product of $\widehat{\mathbb{A}}$. 
\begin{definition}
A normal monomorphism in $\widehat{\mathbb{A}}$ is said to be \emph{linear} if it is in the saturated class of boundary inclusions $\partial \mathbb{A}[a] \to \mathbb{A}[a]$ for $a = i_![n]$.
\end{definition}
\item[(AC3)] Let $f \colon A \to B$ and $g \colon K \to L$ be normal monomorphisms in $\widehat{\mathbb{A}}$, then the map
$$A \square K \sqcup_{A \square L} B \square L \to B \square K$$
is again a normal monomorphism whenever one of them is linear.
\end{itemize}
We will usually use $a \in \mathbb{A}$ for a typical element of an augmentation category.
\end{definition}

Clearly $\Delta$ itself is the prototypical example of an augmentation category, and is minimal in the sense that any other augmentation category will factor through it.

\section{Homotopy of Augmentation Categories}\label{sec:4}

From now on, we will assume that all categories in question are augmentation categories.  This section will be devoted to proving the existence of a model structure on the presheaf category $\widehat{\mathbb{A}}$ for a given augmentation category $\mathbb{A}$.  In this model, the fibrant objects are generalisations of Kan complexes.  The existence of this model structure strongly hinges on the fact that $\widehat{\mathbb{A}}$ is \emph{weakly enriched in simplicial sets}.  In particular, using the tensor product $\square$ and the compatibility condition of (AC2) define for $K \in \widehat{\Delta}$ and $X,Y \in \widehat{\mathbb{A}}$: 
\begin{align*}
\underline{\text{hom}}^\Delta_n(X,Y) &= \text{Hom}_{\widehat{\mathbb{A}}} (X \square i_! (\Delta[n]),Y),\\
X \square K &=  X \square i_! (K), \\
(Y^K)_a &= \text{Hom}_{\widehat{\mathbb{A}}} (\mathbb{A}[a] \square i_! (K) , Y).
\end{align*}

The method of constructing the model structure will, in part, utilise the above simplicial compatibility and $\widehat{\Delta}_\text{Kan}$ to explicitly describe the weak equivalences.

\begin{remark}
The material presented in this section draws heavily on the construction of the stable model structure for dendroidal sets \cite{basicthesis}, which, in turn, follows the presentation of \cite{MR2778589} and \cite{MR3545944}.  In fact, one sees that the definition of the augmentation category is rigid enough that the arguments relating to the model structure developed in \cite{basicthesis} are simply altered in a consistent manner, replacing instances of $\Omega$ with $\mathbb{A}$.  
\end{remark}

\subsection{Normal Monomorphisms}

In this section we look at the properties of the normal monomorphisms as introduced in Definition \ref{thecofibrations}, we recall this definition here using the language of augmentation categories.

\begin{definition}
Let $\mathbb{A}$ be an augmentation category.  A \emph{normal monomorphism} in $\widehat{\mathbb{A}}$ is a map $f\colon X \to Y$ such that $f$ is monic and for each object $a$ of $\mathbb{A}$ and each non-degenerate element $y \in Y_a \backslash f(X)_a$, the isotropy group $\{ g \in \text{Aut} (a) \mid g^\ast(y)= y \}$ is trivial. 
\end{definition}

\begin{remark}\label{ifstrictthenmomo}
Note that if $\mathbb{A}$ has a strict Reedy structure, then the class of normal monomorphisms coincides with the class of monomorphisms.
\end{remark}

Recall that  a class of morphisms is \emph{saturated} if it is closed under retracts, transfinite compositions and pushouts. Similar to the simplicial case, the normal monomorphisms can be described as the saturated class of boundary inclusions.  The following lemma holds for a wider class of categories than just augmentation categories, as proved in {\cite[Proposition 8.1.35]{MR2294028}}.

\begin{lemma}\label{normalmonocound}
The class of normal monomorphisms is the smallest class of monomorphisms closed under pushouts and transfinite compositions that contains all boundary inclusions $\partial \mathbb{A}[a] \to \mathbb{A}[a]$ for $a \in \mathbb{A}$.
\end{lemma}

Using the definition of normal monomorphisms, we will say that an object $A \in \widehat{\mathbb{A}}$ is \emph{normal} if $0 \to A$ is a normal monomorphism.  From the definition of the normal monomorphisms we get the following trivial property, which leads to an observation of maps between normal objects.

\begin{lemma}
A monomorphism $X \to Y$ of $\mathbb{A}$-sets is normal if and only if for any $a \in \mathbb{A}$, the action of $\text{Aut}(a)$ on $Y(a)-X(a)$ is free.
\end{lemma}

\begin{corollary}\label{normtonorm}
If $f \colon A \to B$ is any morphism of $\mathbb{A}$-sets, and $B$ is normal, then $A$ is also normal.  If $f$ is a monomorphism and $B$ is normal, then $f$ is a normal monomorphism
\end{corollary}

Quillen's small object argument applied to the saturated class of boundary inclusions yields the following.

\begin{corollary}
Every morphism $f \colon X \to Y$ of $\mathbb{A}$-sets can be factored as $f = gh$, $g \colon X \to Z$, $h \colon Z \to Y$, where $g$ is a normal monomorphism and $h$ has the RLP with respect to all normal monomorphisms.
\end{corollary}

\begin{definition}
Let $X \in \widehat{\mathbb{A}}$, a \emph{normalisation} of $X$ is a morphism $X' \to X$ from a normal object $X'$ having the RLP with respect to all normal monomorphisms.  Note that such a normalisation exists for any $X$ due to the factoring of the map $0 \to X$.  

\end{definition}

\subsection{Augmented Kan Complexes}

We shall now use the boundary objects, along with our definition of degeneracy maps in a general EZ-category to introduce the concept of a horn object.

\begin{definition}
Let $f \colon b \to a$ be a face map of $a \in \mathbb{A}$. The $f$-horn of $\mathbb{A}[a]$ is the subobject of $\partial \mathbb{A}[a]$ which excludes the object which factors through $f$.  We denote this object $\Lambda^f\mathbb{A}[a]$.  Explicitly, 
$$\Lambda^f\mathbb{A}[a] = \bigcup_{\substack{g \colon a' \to a \\ g \neq f}} g(\mathbb{R}[a']).$$
\end{definition}

\begin{definition}
Let $\mathbb{A}$ be an augmentation category.  An $\mathbb{A}$-Kan complex is an object $X \in \widehat{\mathbb{A}}$ such that it has fillers for all horns. That is there is a lift for all face maps $f$:
$$\xymatrix{\Lambda^f\mathbb{A}[a] \ar[r] \ar[d]& X\ \\
\mathbb{A}[a] \ar@{-->}[ur]}$$
\end{definition}

\begin{remark}
In the simplicial, and indeed the dendroidal settings, there is the concept of an \emph{inner horn}.  In the general setting of augmentation categories, there seems to be no canonical way to define these objects.
\end{remark}

\begin{definition}
The smallest saturated class containing all horn extensions $\Lambda^f\mathbb{A}[a] \to \mathbb{A}[a]$ will be called the class of \emph{anodyne extensions}.  Therefore an object is $\mathbb{A}$-Kan if and only if it has the RLP with respect to all anodyne extensions.
\end{definition}

\begin{proposition}
Let $Z$ be an $\mathbb{A}$-Kan complex and $f \colon A \to B$ a normal monomorphism, then
$$f^\ast \colon \underline{\textup{hom}}^\Delta(B,Z) \to \underline{\textup{hom}}^\Delta(A,Z)$$
is a Kan fibration of simplicial sets.
\end{proposition}

\begin{proof}
We need only prove the result when $f$ is also a boundary inclusion $\partial \mathbb{A}[a] \to \mathbb{A}[a]$ (as the normal monomorphisms are the saturated class of boundary inclusions by Lemma \ref{normalmonocound}).  The map $f^\ast$ has the RLP with respect to the horn inclusion $\Lambda^k[n] \to \Delta[n]$ if and only if $Z$ has the RLP with respect to the map
$$\Lambda^k[n] \square \mathbb{A}[a] \sqcup \Delta[n] \square \partial \mathbb{A}[a] \to \Delta[n] \square \mathbb{A}[a].$$
By (AC3), we have that this map is an anodyne extension, so $Z$ has the RLP with respect to it.
\end{proof}

The following corollary follows by considering the morphism $f \colon 0 \to B$.

\begin{corollary}
If $Z$ is an $\mathbb{A}$-Kan complex, and $B$ a normal object, then $\underline{\text{hom}}^\Delta(B,Z)$ is a Kan complex. 
\end{corollary}

\subsection{Augmented Homotopy}\label{aughomot}

We will use the Kan objects to define a homotopy theory of $\mathbb{A}$-sets.   Due to the compatibility of the tensor structures, we can take $X \square \mathbb{A}[1] \in \widehat{\mathbb{A}}$ to be a cylinder object of $X$.  It comes from the factorisation of the fold map:
$$\xymatrix{X \sqcup X \ar[rr]^{1_X \sqcup 1_X} \ar[dr]_{i_0 \sqcup i_1} && X \\
& X \square \mathbb{A}[1] \ar[ur]_\epsilon}$$
One can see that if $X$ is normal, then $X \square \mathbb{A}[1]$ is normal, and by Corollary \ref{normtonorm}, the map $i_0 \sqcup i_1$ is a normal monomorphism.

\begin{definition}\label{aughomoequiv}
Two morphisms $f,g \colon X \to Y$ in $\widehat{\mathbb{A}}$ are \emph{homotopic} ($f \simeq g$) if there exists $H \colon X \square \mathbb{A}[1] \to Y$ such that $f = Hi_0$ and $g = Hi_1$. That is, the following diagram commutes:
$$\xymatrix{X \ar[r]^-{i_0} \ar[dr]_f & X \square \mathbb{A}[1] \ar[d]^<<<<H & X \ar[l]_-{i_1} \ar[dl]^g \\
& Y}$$
We will say that $f \colon X \to Y$ is a \emph{homotopy equivalence} if there is a morphism $g \colon Y \to X$ such that $fg \simeq 1_Y$ and $gf \simeq 1_X$.
\end{definition}

\begin{definition}\label{weakeqdeff}
A map $f \colon X \to Y$ is a \emph{weak equivalence} if there exists a normalisation (i.e., cofibrant replacement) $f' \colon X' \to Y'$ which induces an equivalence of Kan complexes
$$\underline{\text{hom}}^\Delta(Y',Z) \to \underline{\text{hom}}^\Delta(X',Z)$$
for every $\mathbb{A}$-Kan complex $Z$.  Note that every homotopy equivalence between normal $\mathbb{A}$-sets is a weak equivalence.
\end{definition}

\begin{remark}
As one would expect, if we have a $\mathbb{A}$-set $X$, then the corresponding normalisation $f \colon X' \to X$ is a weak equivalence.
\end{remark}

\begin{lemma}\label{525}
A morphism of $\mathbb{A}$-sets which has the RLP with respect to all normal monomorphisms is a  weak equivalence.
\end{lemma}

\begin{proof}
Let $X \to Y$ be a map between $\mathbb{A}$-sets with the RLP with respect to all normal monomorphisms.  Denote by $Y' \to Y$ the normalisation of $Y$.  Since $Y'$ is normal we may construct a lift
$$
\xymatrix{0 \ar[r] \ar[d] & X \ar[d] \\
Y' \ar[r] \ar@{-->}[ur]|s & Y \rlap{ .}}
$$
We then factor $s$ as a normal monomorphism $i \colon X' \to Y'$ followed by the normalisation $X' \to X$ to get a lift
$$
\xymatrix{Y' \ar@{=}[r] \ar[d]_{i} & Y' \ar[d] \\
X' \ar[r]_{ft} \ar@{-->}[ur]|{f'} & Y\rlap{ .}}
$$
The map $f'i = 1_{Y'}$, and since $Y' \to X'$ is a normal monomorphism, so is $\partial I \square X' \cup I \square Y' \to I \square X'$.  Therefore, there is a lift
 $$
 \xymatrixcolsep{12ex}\xymatrix{
 \partial I \square X' \cup I \square Y'  \ar[r]^-{(if',1_{X'})\cup i_\epsilon} \ar[d] & X'\ar[d] \\
I \square X' \ar@{-->}[ur] \ar[r]_{ft_\epsilon} & X \rlap{ .}
}
$$
We have therefore constructed a homotopy from $if'$ to $1_{X'}$.  Therefore the normalisation $i$ of $f$ is a homotopy equivalence, and therefore induces an equivalence of Kan complexes.
\end{proof}

\subsection{Augmented Trivial Cofibrations}

\begin{definition}
A \emph{trivial cofibration} of $\mathbb{A}$-sets is a cofibration which is also a weak equivalence.
\end{definition}

\begin{lemma}
A pushout of a trivial cofibration is a trivial cofibration
\end{lemma}

\begin{proof}
Let $f \colon A \to B$ be a trivial cofibration and let
$$\xymatrix{A \ar[r] \ar[d] & C \ar[d] \\
B \ar[r] & D}$$
be a pushout square.  We have that normal monomorphisms are closed under pushouts, therefore $C \to D$ is a normal monomorphism.  We need only show that it is also a weak equivalence.  Assume that $A$ and $B$ are normal.  For a $\mathbb{A}$-Kan complex $Z$, we have an induced pullback square
$$\xymatrix{\underline{\text{hom}}^\Delta(D,Z) \ar[r] \ar[d] & \underline{\text{hom}}^\Delta(B,Z) \ar[d] \\
\underline{\text{hom}}^\Delta(C,Z) \ar[r] & \underline{\text{hom}}^\Delta(A,Z) \rlap{ .}}$$
The right side vertical map is a trivial fibration (of simplicial sets) due to the assumption of $A \to B$ being a trivial cofibration between normal objects.  Trivial fibrations are closed under pullbacks and therefore the left vertical map is also a trivial fibration.  Therefore we have shown that $C \to D$ is a trivial cofibration for the case of $A$ and $B$ normal.
Now we assume that $A$ and $B$ are not normal.  The following method is called the \emph{cube argument} (\cite[Lemma 5.3.2]{basicthesis}).  Let $D' \to D$ be a normalisation of $D$ and consider the commutative diagram
$$
\xymatrix{A' \ar[rr] \ar[dr] \ar[dd] & & C' \ar[dr] \ar[dd]|\hole \\
 & B'  \ar[rr] \ar[dd]& & D' \ar[dd]\\
A   \ar[rr]|\hole \ar[dr] & & C \ar[dr] \\
& B  \ar[rr] & & D}
$$
such that the vertical squares are pullbacks.  As we are working in a presheaf category, the pullback of a monomorphism is also a monomorphism.  We have that $D'$ is normal, and therefore by Corollary \ref{normtonorm} we have that the maps $A' \to B'$ and $C' \to D'$ are normal monomorphisms and as such, all vertical maps are normalisations.  The top square is a pushout as pullbacks preserve pushouts in any presheaf category.  Therefore $A' \to B'$ is a trivial cofibration between normal objects, and we have already shown that $C' \to D'$ is also a trivial cofibration.  Therefore we have $C \to D$ has a normalisation which is a stable weak equivalence, and is therefore itself a stable weak equivalence.
\end{proof}

\begin{lemma}\label{anodaretriv}
Anodyne extensions are trivial cofibrations.
\end{lemma}

\begin{proof}
We need only show that every horn inclusion is a weak equivalence.  Let $\Lambda^f \mathbb{A}[a] \to \mathbb{A}[a]$ be such a horn inclusion and $\partial \Delta [n] \to \Delta[n]$  simplicial boundary inclusion.  We have by (AC3) that the map
$$\partial \Delta[n] \square \mathbb{A}[a] \sqcup \Delta[n] \square \Lambda^f\mathbb{A}[a] \to \Delta[n] \square \mathbb{A}[a]$$ 
is an anodyne extension, so every $\mathbb{A}$-Kan complex $Z$ has the RLP with respect to it.  Therefore the map $\underline{\text{hom}}^\Delta(\mathbb{A}[a],Z) \to \underline{\text{hom}}^\Delta(\Lambda^f\mathbb{A}[a],Z)$ is a trivial fibration of simplicial sets under the assumption that $Z$ is Kan.  Therefore $\Lambda^f \mathbb{A}[a] \to \mathbb{A}[a]$ is a weak equivalence and the result is proven.  
\end{proof}

\begin{lemma}\label{lemma524}
Every trivial cofibration is a retract of a pushout of a trivial cofibration between normal objects.
\end{lemma}

\begin{proof}
Let $u \colon A \to B$ be a trivial cofibration, and $A' \to A$ a normalisation of $A$.  We consider the following commutative diagram
$$\xymatrix{A' \ar[r]^{u'} \ar[d] & B' \ar[d] \\
A \ar[r]_u & B}$$
constructed by factoring $A' \to B$ as a normal monomorphism $A' \to B'$ followed by a normalisation of $B$.  As normalisations are weak equivalences, and weak equivalences satisfy the two out of three property, (as the weak equivalences in $\widehat{\Delta}_\text{Kan}$ do so), we have that $A' \to B'$ is a trivial cofibration between normal objects.  We now consider the pushout
$$\xymatrix{A' \ar[r]^u \ar[d] & B' \ar[d] \\
A \ar[r]_v & P}$$
which provides a map $s \colon P \to B$.  We need to show that $s$ has the RLP with respect to the normal monomorphisms, as this would ensure that $u$ is a retract of $v$ via the lifting
$$\xymatrix{A \ar[r]^v \ar[d]_u& P \ar[d]^s \\
B \ar@{=}[r] \ar@{-->}[ur]& B \rlap{ .}}$$
Therefore, we consider the lifting problem
$$\xymatrix{\partial \mathbb{A}[a] \ar[r] \ar[d] & P \ar[d]^s \\
\mathbb{A}[a] \ar[r] & B\rlap{ .}}$$
Using the \emph{cube argument} again, we pullback along $\partial \mathbb{A} [a] \to P$ to form the cube:
$$
\xymatrix{E \ar[rr] \ar[dr] \ar[dd] & & D \ar[dr] \ar[dd]|\hole \\
 & C \ar[rr] \ar[dd]& & \partial \mathbb{A}[a] \ar[dd]\\
A'   \ar[rr]|\hole \ar[dr] & & B' \ar[dr] \\
& A  \ar[rr] & & P \rlap{ ,}}
$$
where the horizontal faces are pushouts and the vertical faces are pullbacks.  We have that $E \to C$ is a normalisation, and therefore all the objects in the top face are normal.  Therefore $E \to C$ has a section and therefore so does the pushout $D \to \partial \mathbb{A}[a]$.  Using this section, we are able to form a commutative diagram
$$\xymatrix{\partial \mathbb{A}[a] \ar[d] \ar[r] & D \ar[r] & B' \ar[d] \\
\mathbb{A}[a] \ar@{-->}[urr] \ar[rr] && B}$$
in which the lift exists.  This therefore gives a solution to the required lifting problem.
\end{proof}

We say that a $\mathbb{A}$-set $X$ is \emph{countable} if each $X(a)$ is a countable set.

\begin{lemma}[{\cite[Proposition 5.3.8]{basicthesis}}]\label{normalsuff}
The class of trivial cofibrations is generated by the trivial cofibrations between countable and normal objects.
\end{lemma}

\subsection{The Model Structure}

\begin{definition}\label{fibdef}
A morphism in $\widehat{\mathbb{A}}$ is a \emph{fibration} if it has the RLP with respect to the trivial cofibrations.
\end{definition}

\begin{theorem}
There is a cofibrantly generated model structure on $\widehat{\mathbb{A}}$ with the defined class of weak equivalences (Definition \ref{weakeqdeff}), fibrations (Definition \ref{fibdef}) and cofibrations (Definition \ref{thecofibrations}).  We will denote this model structure $\widehat{\mathbb{A}}_\text{Kan}$.
\end{theorem}

\begin{proof}
We will show that the axioms (CM1)-(CM5) hold.  First of all, (CM1) holds automatically as $\widehat{\mathbb{A}}$ is a presheaf category.  (CM2) holds from the definition of weak equivalences and the fact that the weak equivalences in the Quillen model on $\widehat{\Delta}$ satisfy this property.  Axiom (CM3) also holds without much concern.  The first non-trivial axiom to show is (CM5).  The fact that every map can be factored as a cofibration followed by a trivial fibration follows from the small object argument for the set of all boundary inclusions and Lemma \ref{525}.  The factorisation as a fibration followed by a trivial cofibration follows from the small object argument for the set of all trivial cofibrations between normal countable $\mathbb{A}$-sets by Lemma \ref{normalsuff}.

One half of (CM4) holds from the definition of the fibrations (Definition \ref{fibdef}).  Assume we have the following commutative digram:
$$\xymatrix{A \ar[r] \ar[d]_i & X \ar[d]^p \\
B \ar[r] & Y \rlap{ ,}}$$  
for $i$ a cofibration and $p$ a trivial fibration.  We need to produce a lift $B \to X$.  We factor $p \colon X \to Y$ as a cofibration $Y \to Z$ followed by a map $Z \to X$ having the RLP with respect to all cofibrations.  Then $Z \to X$ is a weak equivalence and by two out of three, so is $Y \to Z$.  We first find the lift in
$$\xymatrix{A \ar[r] \ar[d]_i & X \ar[r] & Z \ar[d] \\
B \ar[rr] \ar@{-->}[urr] && Y \rlap{ ,}}$$
and then in
$$\xymatrix{X \ar@{=}[r] \ar[d] & X \ar[d]^p \\
Z \ar@{-->}[ur] \ar[r] & Y \rlap{ .}}$$
The composition of the two lifts gives us the necessary lift $B \to X$.    The generating cofibrations are given by the boundary inclusions and the set of generating trivial cofibrations are the trivial cofibrations between normal and countable $\mathbb{A}$-sets. 
\end{proof}

\begin{proposition}
The fibrant objects in $\widehat{\mathbb{A}}_\text{Kan}$ are the $\mathbb{A}$-Kan complexes.
\end{proposition}

\begin{proof}
Let $Z$ be fibrant in $\widehat{\mathbb{A}}_\text{Kan}$, by Lemma \ref{anodaretriv} we have that the anodyne extensions are exactly the trivial cofibrations, therefore $Z$ has the RLP with respect to the anodyne extensions and is therefore $\mathbb{A}$-Kan.
Conversely, let $Z$ be $\mathbb{A}$-Kan, and $A \to B$ a trivial cofibration between normal objects, then the map $\underline{\text{hom}}^\Delta(B,Z) \to \underline{\text{hom}}^\Delta(A,Z)$ is a trivial fibration of simplicial sets.  We have that the trivial fibrations in $\widehat{\Delta}_\text{Kan}$ are surjective on vertices, and we can deduce that $Z$ has the RLP with respect to the map $A \to B$.  Lemma \ref{lemma524} then implies that every $\mathbb{A}$-Kan complex has the RLP with respect to all trivial cofibrations.
\end{proof}

\subsection{Properties of $\text{Ho}(\widehat{\mathbb{A}}_\text{Kan})$}

In this section we briefly list some of the required properties of $\widehat{\mathbb{A}}_\text{Kan}$.  The first property that we prove is left-properness which will be essential when considering augmented presheaves.

\begin{proposition}
$\widehat{\mathbb{A}}_\text{Kan}$ is left-proper.
\end{proposition}

\begin{proof}
Consider the following pushout
$$\xymatrix{A \ar[r] \ar[d] & C \ar[d] \\
B \ar[r] & D \rlap{ ,}}$$
with $A \to B$ a weak equivalence and $A \to C$ a cofibration.  We can reduce to the case where all objects are normal via the \emph{cube argument}.  We have an induced diagram
$$\xymatrix{\underline{\text{hom}}^\Delta(D,Z) \ar[r] \ar[d] & \underline{\text{hom}}^\Delta(B,Z) \ar[d] \\
\underline{\text{hom}}^\Delta(C,Z) \ar[r] & \underline{\text{hom}}^\Delta(A,Z) \rlap{ ,}}$$
which is a pullback for any $Z \in \widehat{\mathbb{A}}$.  If $Z$ were a $\mathbb{A}$-Kan complex, then all simplicial sets in the above diagram are also Kan complexes and the right vertical map is an equivalence of simplicial sets.  The left vertical map is an equivalence.  This follows as trivial fibrations are stable under pullbacks, and so by Ken Brown's Lemma \cite{MR0341469}, all weak equivalences between fibrant objects are stable under pullback.
\end{proof}

The following lemma holds from the construction of compatible Kan complex objects.

\begin{lemma}\label{quiladj1}
There is a Quillen adjunction
$$i_! : \widehat{\Delta}_\text{Kan} \rightleftarrows \widehat{\mathbb{A}}_\text{Kan} : i^\ast.$$
\end{lemma}

\begin{proof}
The right adjoint $i^\ast$ sends fibrations (resp., trivial fibrations) of $\widehat{\mathbb{A}}_\text{Kan}$ to fibrations (resp., trivial fibrations) of $\widehat{\Delta}_\text{Kan}$ by construction, and therefore is part of a Quillen pair, for which $i_!$ is the left adjoint.
\end{proof}

Recall that we do not always have the assumption that $\widehat{\mathbb{A}}_\text{Kan}$ satisfies the pushout-product axiom, therefore we cannot hope for a closed monoidal structure in full generality.  However, what does hold is the fact that $\text{Ho}(\widehat{\mathbb{A}}_\text{Kan})$ is a simplicial category.  Using the tools presented in \cite[Proposition 5.4.4]{basicthesis} and \cite[Lemma 3.8.1]{MR3545944} we can formally prove the following, although the result is non-surprising due to the way we have constructed the classes of maps in our model structure.

\begin{lemma}
The category $\textup{Ho}(\widehat{\mathbb{A}}_\text{Kan})$ is enriched over $\textup{Ho}(\widehat{\Delta}_\text{Kan})$.
\end{lemma}

\begin{remark}
If $\mathbb{A}$ is a strict EZ-category then the model $\widehat{\mathbb{A}}_\text{Kan}$ is a Cisinski type model structure defined in \cite{MR2294028}.  This follows as for a strict EZ-category we have by Remark \ref{ifstrictthenmomo} that the normal monomorphisms are then just the monomorphisms.
\end{remark}

\section{Augmented Homotopical Algebraic Geometry}\label{sec:5}

\subsection{Local Model Structure on Augmented Presheaves}\label{auglocmodels}

We will now build a local model structure on $\mathbb{A} \textbf{-Pr} (\mathcal{C}) := {\widehat{\mathbb{A}}}^{\mathcal{C}^{op}}$ which reflects the local model structure on simplicial presheaves.  First of all, we consider the projective point-wise model structure on $\mathbb{A} \textbf{-Pr} (\mathcal{C})$, again denoted $\mathbb{A} \textbf{-Pr}_\text{proj} (\mathcal{C})$.  Note that as $\widehat{\mathbb{A}}$ is left-proper and combinatorial (for it is accessible and cofibrantly generated), we have that the projective point-wise model is also left-proper and combinatorial.  

We will now construct the analogue of hypercovers.  In the simplicial case we introduced hypercovers using the boundary construction, which was done as we can simply edit the necessary components to get a class of augmented hypercovers.

\begin{definition}\label{aughyper}
Let $\mathbb{A}$ be an augmentation category and $(\mathcal{C},\tau)$ a site. A map $f \colon X \to Y$ in $\mathbb{A} \textbf{-Pr}(\mathcal{C})$ is a \emph{hypercovering} if for all $a \in \mathbb{A}$ the map
$$X_a \to \text{Hom}_{\widehat{\mathbb{A}}}(\partial \mathbb{A}[a] , X) \times_{\text{Hom}_{\widehat{\mathbb{A}}}(\partial \mathbb{A}[a] , Y)} Y_a$$
is a $\tau$-covering in $\textup{Ho}(\mathbb{A} \textbf{-Pr}_\text{proj}(\mathcal{C}))$.  Note that by using Lemma \ref{coskrep}, this is equivalent to asking the same condition for the class of maps
$$X_a \to (\textup{cosk}_{d(a)-1}X)_r \times_{(\textup{cosk}_{d(a)-1}X)_r} Y_a.$$
\end{definition}

\begin{definition}
The \emph{local model structure} on $\mathbb{A} \textbf{-Pr} (\mathcal{C})$ is the left Bousfield localisation of $\mathbb{A} \textbf{-Pr}_\text{proj} (\mathcal{C})$ at the class of hypercoverings.  We shall denote this model structure $\mathbb{A} \textbf{-Pr}_\tau(\mathcal{C})$.
\end{definition}

\begin{definition}
For a site $(\mathcal{C}, \tau)$, we will call the homotopy category $\text{Ho}(\mathbb{A} \textbf{-Pr}_\tau(\mathcal{C}))$ the category of \emph{$\mathbb{A}$-augmented stacks on $\mathcal{C}$}. 
\end{definition}

\begin{lemma}
For $(\mathcal{C},\tau)$ a site, there is a Quillen adjunction
$$i_! : \textbf{sPr}_\tau(\mathcal{C}) \rightleftarrows \mathbb{A} \textbf{-Pr}_\tau(\mathcal{C}) : i^\ast .$$
\end{lemma}

\begin{proof}
We will prove the statement for the respective injective local models and then we can compose with the identity functor to the projective local case, which will prove the result.  In the local injective models, the cofibrations are the point-wise cofibrations.  By Lemma \ref{quiladj1}, we know that $i_!$ sends cofibrations of $\widehat{\Delta}_\text{Kan}$ to cofibrations of $\widehat{\mathbb{A}}_\text{Kan}$.  Therefore, we now need only show that $i_!$ preserves the trivial cofibrations now.  This follows as if $f \colon X \to Y$ is a (non-augmented) hypercover, then  $i_!f \colon i_!X \to i_!Y$ is an augmented hypercover.
\end{proof}

\begin{remark}
We have introduced the above theory for the case when $\mathcal{C}$ is a simplicial site.  However for $\mathbb{A}$-presheaves, it would also make sense to allow \emph{$\mathbb{A}$-sites}.  That is, categories $\mathcal{C}$ enriched over $\widehat{\mathbb{A}}$ such that there is a Grothendieck topology on $\pi_0(\mathcal{C})$.
\end{remark}

\subsection{Local Weak Equivalences}

We now describe what the weak equivalences in the local model structure should look like.  We introduce the concept of augmented homotopy groups, resembling the simplicial case, using the augmented homotopy of Section \ref{aughomot}.  

\begin{definition}\label{homotopygrps}
Let $X \in \widehat{\mathbb{A}}$ be fibrant in $\widehat{\mathbb{A}}_\text{Kan}$, and $a \in \mathbb{A}$.  Denote by $\pi_a(X,x_0)$ the set of equivalence classes of morphisms $\alpha \colon \mathbb{A}[a] \to X$ which fit into the following commutative diagram in $\widehat{\mathbb{A}}$:
$$\xymatrix{\partial \mathbb{A}[a] \ar[r] \ar[d] & \mathbb{A}[0] \ar[d]^{x_0} \\
\mathbb{A}[a] \ar[r]_\alpha & X  \rlap{ ,}}$$
where the equivalence relation is given by homotopy equivalence of Definition \ref{aughomoequiv}.
\end{definition}

\begin{remark}
Although we will not pursue it here, one would hope that the object $\pi_a(X,x_0)$ is a group which is abelian for $d(a) \geq 2$.
\end{remark}

\begin{definition}
Let $(\mathcal{C}, \tau)$ be a site.  A map $f \colon \mathcal{F} \to \mathcal{F}'$ in $\mathbb{A} \textbf{-Pr}(\mathcal{C})$ is a \textit{local weak equivalence} if:
\begin{itemize}
\item The induced map $\pi_0^\tau \mathcal{F} \to \pi_0^\tau\mathcal{F}'$ is an isomorphism of sheaves.
\item Squares of following form are pullbacks after sheafification:
$$
\xymatrix{\pi_a \mathcal{F} \ar[r] \ar[d] &\ar[d] \pi_a \mathcal{F}' \\
\mathcal{F}_0 \ar[r] &  \mathcal{F}'_0 \rlap{ .}}
$$
\end{itemize}
\end{definition}

\begin{conjecture}\label{firstconjecture}
A map $f \colon \mathcal{F} \to \mathcal{F}'$ is a local weak equivalence if and only if it is a hypercover.
\end{conjecture}

If the above conjecture is true, which seems likely due to the combinatorial nature of augmentation categories in relation to the simplex category, then we would get the following result, mirroring Theorem \ref{projlocmodel}.

\begin{corollary}
Let $(\mathcal{C}, \tau)$ be a site.  There exists a cofibrantly generated model structure on the category $\mathbb{A} \textbf{-Pr}(\mathcal{C})$ where a map $f \colon \mathcal{F} \to \mathcal{F}'$ is a:
\begin{enumerate}
\item Weak equivalence if it is a local weak equivalence.
\item Fibration if it has the RLP with respect to the trivial cofibrations.
\item Cofibration if it is a cofibration in the point-wise projective model.
\end{enumerate}
Moreover this model is Quillen equivalent to the local model $\mathbb{A} \textbf{-Pr}_\tau(\mathcal{C})$.
\end{corollary}

\subsection{Enriched Structure}

In this section we will show that the local model structure $\mathbb{A} \textbf{-Pr}_\tau(\mathcal{C})$ is enriched over the local model structure on simplicial presheaves, and is therefore a simplicial category.  In the case that $\widehat{\mathbb{A}}$ is a closed monoidal category we can go further and show that $\mathbb{A} \textbf{-Pr}_\tau(\mathcal{C})$ is in fact a closed monoidal model category.  To do this, we will use a trick used in \cite[\S 3.6]{MR2137288}, which considers instead the local injective model structure.  Denote by $\mathbb{A} \textbf{-Pr}_\textup{inj}(\mathcal{C})$ the injective model structure.  From the properties of $\widehat{\mathbb{A}}_\text{Kan}$ we have that this model is left-proper and cofibrantly generated, and therefore a left Bousfield localisation exists.  We localise at the set of hypercovers once again and retrieve the model category $\mathbb{A} \textbf{-Pr}_{\text{inj,}\tau}(\mathcal{C})$.  Clearly we have an equivalence of categories $\text{Ho}(\mathbb{A} \textbf{-Pr}_{\text{inj,}\tau}(\mathcal{C})) \simeq \text{Ho}(\mathbb{A} \textbf{-Pr}_{\tau}(\mathcal{C}))$ as the weak equivalences in each model are the same.  As the cofibrations in the local injective model are simply the normal monomorphisms, and finite products preserve local weak equivalences, anything we have proved regarding the enriched structure of $\widehat{\mathbb{A}}_\text{Kan}$ can be carried over to this setting.  Using this justification, we state the two following lemmas and corresponding definitions.

\begin{lemma}
The category $\textup{Ho}(\mathbb{A} \textbf{-Pr}_{\tau}(\mathcal{C}))$ is enriched over $\textup{Ho}(\textbf{sPr}_\tau(\mathcal{C}))$, and subsequently over $\textup{Ho}(\widehat{\Delta}_\text{Kan})$.  In particular, $\mathbb{A} \textbf{-Pr}_{\tau}(\mathcal{C})$ is a simplicial model category.
\end{lemma}

\begin{definition}
The simplicial presheaf enrichment in $\textup{Ho}(\mathbb{A} \textbf{-Pr}_{\tau}(\mathcal{C}))$ will be denoted
$$\mathbb{A} \text{-} \mathbb{R}_\tau^\Delta \underline{\mathcal{H}\text{om}}(-,-) \colon \textup{Ho}(\mathbb{A} \textbf{-Pr}_{\tau}(\mathcal{C})) \times \textup{Ho}(\mathbb{A} \textbf{-Pr}_{\tau}(\mathcal{C}))  \to  \textup{Ho}(\textbf{sPr}_\tau(\mathcal{C})).$$
The corresponding simplicial enrichment will be denoted
$$\mathbb{A} \text{-} \mathbb{R}_\tau^\Delta \underline{\text{Hom}}(-,-) \colon \textup{Ho}(\mathbb{A} \textbf{-Pr}_{\tau}(\mathcal{C})) \times \textup{Ho}(\mathbb{A} \textbf{-Pr}_{\tau}(\mathcal{C}))  \to \textup{Ho}(\widehat{\Delta}_\text{Kan}).$$
\end{definition}

\begin{lemma}
If the presheaf category $\widehat{\mathbb{A}}$ satisfies the pushout-product axiom, then the category $\textup{Ho}(\mathbb{A} \textbf{-Pr}_{\tau}(\mathcal{C}))$ is a closed monoidal category, and is subsequently enriched over $\text{Ho}(\widehat{\mathbb{A}}_\text{Kan})$.  In particular, $\mathbb{A} \textbf{-Pr}_{\tau}(\mathcal{C})$ is a  $\mathbb{A}$-model category.

\end{lemma}

\begin{definition}
The internal-hom in $\textup{Ho}(\mathbb{A} \textbf{-Pr}_{\tau}(\mathcal{C}))$ will be denoted
$$\mathbb{A} \text{-} \mathbb{R}_\tau \underline{\mathcal{H}\text{om}}(-,-) \colon \textup{Ho}(\mathbb{A} \textbf{-Pr}_{\tau}(\mathcal{C})) \times \textup{Ho}(\mathbb{A} \textbf{-Pr}_{\tau}(\mathcal{C}))  \to \textup{Ho}(\mathbb{A} \textbf{-Pr}_{\tau}(\mathcal{C})).$$
The corresponding augmented enrichment will be denoted
$$\mathbb{A} \text{-} \mathbb{R}_\tau \underline{\text{Hom}}(-,-) \colon \textup{Ho}(\mathbb{A} \textbf{-Pr}_{\tau}(\mathcal{C})) \times \textup{Ho}(\mathbb{A} \textbf{-Pr}_{\tau}(\mathcal{C}))  \to \textup{Ho}(\widehat{\mathbb{A}}_\text{Kan}).$$
\end{definition}

\subsection{Augmented Derived Stacks}

We shall now use the model structures developed in Section \ref{auglocmodels} to discuss the theory of augmented stacks, which will be the main objects of interest in this article.  Following the simplicial setting, we make the following definition.

\begin{definition}
Let $(\mathcal{C},\tau)$ be a site, and $\mathbb{A}$ an augmentation category.  
\begin{itemize}
\item The category $\textup{Ho}(\mathbb{A} \textbf{-Pr}_{\tau}(\mathcal{C}))$ will be called the \emph{category of $\mathbb{A}$-augmented $\infty$-stacks}. 
\item An object $F \in \textup{Ho}(\mathbb{A} \textbf{-Pr}_{\tau}(\mathcal{C}))$ will be referred to as a \emph{$\mathbb{A}$-augmented $\infty$-stack}.  
\item For two augmented stacks $F,G \in \textup{Ho}(\mathbb{A} \textbf{-Pr}_{\tau}(\mathcal{C}))$, we will call the object $\mathbb{A} \text{-} \mathbb{R}_\tau \underline{\mathcal{H}\text{om}}(-,-)$ (resp., $\mathbb{A} \text{-} \mathbb{R}_\tau^\Delta \underline{\mathcal{H}\text{om}}(-,-)$) the \emph{$\mathbb{A}$-mapping stack} (resp., mapping stack) from $F$ to $G$. Note that the $\mathbb{A}$-mapping stack does not exist in full generality.
\end{itemize}
\end{definition}

\subsection{Augmented Geometric Derived Stacks}

We now introduce augmented derived geometric stacks through a modified $n$-hypergroupoid construction.  Of course, all that we do in this setting can be reformulated for when the site in question is not $\textbf{dAff}$, but any homotopical algebraic geometric context.

\begin{definition}
A derived Artin (resp., derived Deligne-Mumford) $(\mathbb{A},n)$-hypergroupoid is an object $X \in \textbf{dAff}^{\mathbb{A}^{op}}$ such that the maps
$$X_a = \text{Hom}_{\widehat{\mathbb{A}}}(\mathbb{A}[a],X) \to \text{Hom}_{\widehat{\mathbb{A}}}(\Lambda^f\mathbb{A}[a],X)$$
are smooth (resp., \'{e}tale) surjections for all objects $a$ and face maps $f$, and  are isomorphisms for all $a$ with $d(a) > n$.  Note that a derived $(\mathbb{A},n)$-hypergroupoid is an augmented $\mathbb{A}$-Kan complex.  Moreover, using Lemma \ref{coskrep} we see that a $(\mathbb{A},n)$-hypergroupoid is $(n+1)$-coskeletal.
\end{definition}

\begin{definition}
A derived Artin (resp., derived Deligne-Mumford) trivial $(\mathbb{A},n)$-hypergroupoid is a map $f \colon X \to Y$ in $\textbf{dAff}^{\mathbb{A}^{op}}$ such that the maps
$$X_a \to \text{Hom}_{\widehat{\mathbb{A}}}(\partial \mathbb{A}[a] , X) \times_{\text{Hom}_{\widehat{\mathbb{A}}}(\partial \mathbb{A}[a] , Y)} Y_a$$
are smooth (resp.,  \'{e}tale) surjections for all $a,f$ and are isomorphisms for all $a$ with $d(a) > n$.  In particular, a derived Artin trivial $(\mathbb{A},n)$-hypergroupoid is a trivial fibration in $\widehat{\mathbb{A}}_\text{Kan}$.
\end{definition}

\begin{lemma}
Let $f \colon X \to Y$ be a trivial $(\mathbb{A},n)$-hypergroupoid then we have that $X = Y \times_{\textup{cosk}_{n-1} Y} \textup{cosk}_{n-1}X$.
\end{lemma}

\begin{proof}
This follows from comparing the definition to the general result about augmented coskeletal objects in Lemma \ref{coskrep}.  Note that this shows that $(\mathbb{A},n)$-hypergroupoids can be seen as $n$-truncated $\mathbb{A}$-hypercovers.
\end{proof}

\begin{definition}
A model for the $\infty$-category of strongly quasi-compact $(\mathbb{A},n)$-geometric derived Artin (resp., Deligne-Mumford) stacks is given by the relative category consisting of the derived Artin (resp., Deligne-Mumford) $(\mathbb{A},n)$-hypergroupoids and the class of derived trivial Artin (resp., Deligne-Mumford) $(\mathbb{A},n)$-hypergroupoids.  We will denote the homotopy category as 
$$\mathbb{A} \text{-} \mathcal{G}_n^\text{sm}(\textbf{dAff}) \qquad \left( \text{resp., } \mathbb{A} \text{-} \mathcal{G}_n^\text{\'{e}t}(\textbf{dAff}) \right)$$
\end{definition}

\begin{remark}
We could have also (equivalently) formulated the theory of $(\mathbb{A},n)$-geometric derived Artin or Deligne-Mumford stacks by using a representability criteria.  However, using the hypergroupoid construction highlights the beauty of the construction with the intertwining of the hypercovering conditions.
\end{remark}

We have listed the following as a conjecture, as relative categories do not have the notion of Quillen equivalence, and therefore we can only talk about the derived adjunction that may exist between the homotopy categories.  To prove the following conjecture, one would need to construct the full Quillen model structure, and prove that there is a Quillen adjunction.  It could also be the case that this result does not hold in all cases, but only for a subclass of augmentation categories.

\begin{conjecture}\label{conj3}
Let $\mathbb{A}$ be an augmentation category, then there is an adjunction
\begin{align*}
\mathbb{L} i_! :  \mathcal{G}_n^\text{sm}(\textbf{dAff})  &\rightleftarrows \mathbb{A} \text{-} \mathcal{G}_n^\text{sm}(\textbf{dAff})  : \mathbb{R}i^\ast\\
\bigg( \text{resp., } \mathbb{L} i_! :  \mathcal{G}_n^\text{\'{e}t}(\textbf{dAff})  &\rightleftarrows \mathbb{A} \text{-} \mathcal{G}_n^\text{\'{e}t}(\textbf{dAff})  : \mathbb{R}i^\ast \bigg)
\end{align*}
\end{conjecture}

We finish this section by outlining a potential direction for future research of $(a,n)$-hypergroupoids.  One of the key tools used in derived algebraic geometry is that of quasi-coherent complexes as defined in \cite[\S 5.2]{MR2717174}, of particular interest is the cotangent complex which controls infinitesimal deformations.  In \cite{MR3033634} an alternative, homotopy equivalent definition of quasi-coherent complexes was given using $n$-hypergroupoids, giving a cosimplicial object.  This definition would be adjustable to the augmented setting, giving rise to a particular coaugmented object.  The question is then what type of deformation would such an augmented cotangent complex control?

\section{Equivariant Stacks}\label{sec:6}

Our first non-trivial example of an augmentation category will be crossed simplicial groups in the sense of Definition~\ref{csgrpdeff},  introduced by Loday and Fiedorowicz~\cite{loday} (and independently by Krasauskas under the name of \textit{skew-simplicial sets}~\cite{krasauskas}).  In particular a \textit{crossed simplicial group} is a category $\Delta \mathfrak{G}$ equipped with an embedding $i \colon  \Delta \hookrightarrow \Delta \mathfrak{G}$ such that:
\begin{enumerate}
\item The functor $i$ is bijective on objects.
\item Any morphism $u \colon  i[m] \to i[n]$ in $\Delta \mathfrak{G}$ can be uniquely written as $i(\phi) \circ g$ where $\phi \colon  [m] \to [n]$ is a morphism in $\Delta$ and $g$ is an automorphism of $i[m]$ in $\Delta \mathfrak{G}$.
\end{enumerate}

The canonical example of such a category is the cyclic category of Connes, which we will denote $\Delta \mathfrak{C}$~\cite{connes1}.  The following lemma will be used to prove several results in the rest of this article, and follows from Proposition~\ref{genreedycrossed}.

\begin{lemma}[{\cite[\S 6]{genreedy}}]\label{crossisez}
Let $\mathbb{R}$ be a strict EZ-category.  Then any crossed group $\mathbb{R} \mathfrak{G}$ on $\mathbb{R}$ is an EZ-category.
\end{lemma}

\begin{proposition}
The category $\Delta \mathfrak{G}$ is an augmentation category.
\end{proposition}

\begin{proof}
We check that $\Delta \mathfrak{G}$ satisfies the conditions (AC1) -- (AC3) appearing in Definition~\ref{augdef}:
\begin{enumerate}
\item[(AC1)] From Lemma~\ref{crossisez} we see that $\Delta \mathfrak{G}$ is an EZ-category.  A tensor structure is given by just taking the cartesian product. 
\item[(AC2)] The inclusion $\Delta \hookrightarrow \Delta \mathfrak{G}$ and the compatibility of the monoidal structures follows by definition.
\item[(AC3)] In Lemma~\ref{ismonoidalclosed} below, we prove that the normal monomorphisms have the pushout-product property in full generality, which is a stronger result.  We forgo the proof of (AC3) and instead refer to the proof of Lemma~\ref{ismonoidalclosed}.
model structure is closed monoidal, which is a stronger result.
\end{enumerate}
\end{proof}

We describe the adjoint $i_!$, which appears for the cyclic case in \cite[\S 3.1]{connes2}.   As the functor $i^\ast \colon \widehat{\Delta \mathfrak{C}} \to \widehat{\Delta}$ forgets the $\Delta \mathfrak{G}$ action, the left adjoint is naturally a free construction.  In particular, for a simplicial set $X$ we have $i_!(X)_n = \text{Aut}_{\Delta \mathfrak{G}}([n]) \times X_n $, where the  $\text{Aut}_{\Delta \mathfrak{G}}([n])$ action given by left multiplication.  The face and degeneracy maps are given by: 
\begin{align*}
d_i(g,x)&=\left(d_i(g),d_{g^{-1}(i)}(x)\right),\\
s_i(g,x)&=\left(s_i(g),s_{g^{-1}(i)}(x)\right).
\end{align*}

\begin{corollary}\label{equivariantcorr}\leavevmode
\begin{itemize}
\item There is a Quillen model structure on the category of ${\Delta \mathfrak{G}}$-sets, denoted $\widehat{{\Delta \mathfrak{G}}}_\text{Kan}$, where the fibrant objects are the ${\Delta \mathfrak{G}}$-Kan complexes, and the cofibrations are the normal monomorphisms.  Moreover there is a Quillen adjunction 
$$i_! \colon \widehat{\Delta}_\text{Kan} \rightleftarrows \widehat{{\Delta \mathfrak{G}}}_\text{Kan} \colon i^\ast.$$
\item There is a local model structure on the category of ${\Delta \mathfrak{G}}$-presheaves, denoted ${\Delta \mathfrak{G}} \textbf{-Pr}_\tau(\mathcal{C})$ obtained as the left Bousfield localisation of the point-wise Kan model at the class of ${\Delta \mathfrak{G}}$-hypercovers. Moreover there is a Quillen adjunction
$$i_! : \textbf{sPr}_\tau(\mathcal{C}) \rightleftarrows {\Delta \mathfrak{G}} \textbf{-Pr}_\tau(\mathcal{C}) : i^\ast.$$
\item The category of $({\Delta \mathfrak{G}},n)$-geometric derived Artin stacks, which we denote ${\Delta \mathfrak{G}} \text{-} \mathcal{G}_n^\text{sm}(\textbf{dAff})$, is given as the homotopy category of derived $({\Delta \mathfrak{G}},n)$-hypergroupoids with respect to the class of trivial $({\Delta \mathfrak{G}},n)$-hypergroupoids.
\end{itemize}
\end{corollary}

\begin{definition}
The augmented homotopical algebraic geometry theory arising from ${\Delta \mathfrak{G}}$ will be referred to as \emph{$\mathfrak{G}$-equivariant}.  For example, the category $\text{Ho}({\Delta \mathfrak{G}} \textbf{-Pr}_\tau(\mathcal{C}))$ will be referred to as the category of \emph{$\mathfrak{G}$-equivariant $\infty$-stacks}.
\end{definition}

We will now explore the Kan model structure on $\widehat{{\Delta \mathfrak{G}}}$, which will lead to the justification of using the term $\mathfrak{G}$-equivariant to describe the augmentation.   We shall fix our attention to the cyclic category ${\Delta \mathfrak{C}}$, however, the homotopy of arbitrary crossed simplicial groups can be discussed, see~\cite{balchin2} for details.

The Kan model structure on $\widehat{{\Delta \mathfrak{C}}}$ was in fact the first model structure to be developed for cyclic sets by Dwyer, Hopkins and Kan~\cite{homotopycc}.  

\begin{proposition}[{\cite[Theorem 3.1]{homotopycc}}]\label{csweak}
The category $\widehat{{\Delta \mathfrak{C}}}$ has a cofibrantly generated model structure where a map $f \colon X \to Y$ is a:
\begin{itemize}
\item Weak equivalence if $i^\ast(f) \colon i^\ast(X) \to i^\ast(Y)$ is a weak equivalence in $\widehat{\Delta}_\text{Kan}$.
\item Fibration if $i^\ast(f) \colon i^\ast(X) \to i^\ast(Y)$ is a fibration in $\widehat{\Delta}_\text{Kan}$.
\item Cofibration if it has the LLP with respect to the trivial fibrations.
\end{itemize}
We shall denote this model $\widehat{{\Delta \mathfrak{C}}}_{\text{DHK}}$.
\end{proposition}

By comparison of the structure of the fibrations \cite[Proposition 3.2]{homotopycc} and cofibrations \cite[Proposition 3.5]{homotopycc}, we see that the model structure of Proposition~\ref{csweak} is exactly that of $\widehat{{\Delta \mathfrak{C}}}_\text{Kan}$.

\begin{corollary}
There is an equivalence $\widehat{{\Delta \mathfrak{C}}}_\text{Kan} \rightleftarrows \widehat{{\Delta \mathfrak{C}}}_\text{DHK}$.
\end{corollary}

Recall from~\cite[Proposition 2.8]{homotopycc} that there is a cyclic realisation functor $|-|_\mathfrak{C} \colon \widehat{{\Delta \mathfrak{C}}} \to \textbf{Top}^{SO(2)}$ along with its right adjoint $S_\mathfrak{C}(-)$.  We can now use the theory of~\cite{homotopycc} to describe the homotopy type of $\widehat{{\Delta \mathfrak{C}}}_\text{Kan}$.

\begin{proposition}[{\cite[Theorem 2.2]{Dwyer1984147}}]\label{ctopweak}
There is a model structure on $\textbf{Top}^{SO(2)}$ where a map $f \colon X \to Y$ is a:
\begin{itemize}
\item Weak equivalence if the underlying map of topological spaces is a weak equivalence in $\textbf{Top}$.
\item Fibration if the underlying map of topological spaces is a fibration in $\textbf{Top}$.
\item Cofibration if it has the LLP with respect to the trivial fibrations.
\end{itemize}
\end{proposition}

\begin{proposition}[{\cite[Corollary 4.3]{homotopycc}}]\label{cyclicequiv}
There is a Quillen equivalence
$$\widehat{{\Delta \mathfrak{C}}}_\text{Kan} \rightleftarrows \textbf{Top}^{SO(2)},$$
with the equivalence furnished by the cyclic realisation and singular functors.
\end{proposition}

Using the fact that a cyclic object is the same as a simplicial objects along with extra datum, we see that a cyclic $\infty$-stack can be viewed as an $\infty$-stack with extra datum.  In light of Proposition~\ref{cyclicequiv}, we therefore see that objects of $\text{Ho}({\Delta \mathfrak{C}} \textbf{-Pr}_\tau(\mathcal{C}))$ can be viewed as $\infty$-stacks along with an $SO(2)$-action.  Hence the terminology of \emph{equivariant}.

One important property of the Kan model structure on crossed simplicial groups is that the pushout-product axiom holds, as a consequence we have that $ {\Delta \mathfrak{G}} \textbf{-Pr}_\tau(\mathcal{C}))$ is a closed monoidal model category.  We now provide a proof of this claim, adapted from \cite[Lemma 2.2.15]{seteklevmaster}.

\begin{lemma}\label{ismonoidalclosed}
For ${\Delta \mathfrak{G}}$ a crossed simplicial group, $\widehat{{\Delta \mathfrak{G}}}_\text{Kan}$ is a monoidal model category.
\end{lemma}

\begin{proof}
We first show that the pushout-product axiom holds.  We need to show that given any pair of cofibrations $f \colon X \to Y$ and $f' \colon X' \to Y'$, their pushout-product
$$f \boxtimes f' \colon (X \times Y') \bigsqcup_{X \times X'} (Y \times X')\to Y \times Y'$$
is a cofibration that is trivial whenever $f$ or $f'$ is. The condition amounts to considering the following pushout diagram
\begin{equation}\label{cyclicpop}
\xymatrix{X \times X' \ar[r]^{f'_\ast} \ar[d]_{f_\ast} & X \times Y' \ar[d] \ar@/^/[ddr]^{f'^\ast} \\
Y \times X' \ar[r] \ar@/_/[rrd]_{f'^\ast} & P \ar[dr]|-{f \boxtimes f'} \\
&& Y \times Y'}
\end{equation}
By the universal condition of pushouts, we have that $P$ is represented by pairs $(y,x') \in Y \times X'$ and $(x,y') \in X \times Y'$ subject to the relation $(f(x),x') \sim (x,f'(x'))$.  We first show that $f_\ast$ and $f'_\ast$ are cofibrations.  We already have that both of these maps are monomorphisms, furthermore, if $(y,x')$ is not in the image of $f_\ast$, then $y$ is not in the image of $f$, which is the normality condition (this shows that $f_\ast$ is a cofibration, $f'_\ast$ follows similarly).  Next we will show that $f \boxtimes f'$ is a cofibration.  It is clearly a monomorphism as both $f$ and $f'$ are monomorphisms.  Let $p \in Y \times Y'$ be an element such that it is not in the image of $f \boxtimes f'$.  Therefore $p$ is represented by an element not in the image of either $f_\ast$ or $f'_\ast$.  However, since we have that these maps are cofibrations, the $\mathfrak{G}_n$ acts freely on $p$ and $f \boxtimes f'$ is therefore a cofibration.

All that is left to show is the condition about the trivial cofibrations.  However, Diagram \ref{cyclicpop} induces a pushout diagram in $\widehat{\Delta}_\text{Kan}$ via the forgetful functor $i^\ast$.  As $\widehat{\Delta}_\text{Kan}$ itself is a monoidal model category, we can conclude that the pushout-product is a weak equivalence if $f$ or $f'$ is such.

We now need to show that the unit axiom holds.  Let $X$ be a normal object, then we need to show that ${\Delta \mathfrak{G}}[0]' \times X \to {\Delta \mathfrak{G}}[0] \times X$ is a weak equivalence, where ${\Delta \mathfrak{G}}[0]'$ is the normalisation of the unit.  This is true when $i^\ast({\Delta \mathfrak{G}}[0]' \times X) \to i^\ast({\Delta \mathfrak{G}}[0] \times X)$ is a weak equivalence of simplicial sets, and as right adjoints preserve products, this is equivalent to asking that $i^\ast({\Delta \mathfrak{G}}[0]') \times i^\ast(X) \to i^\ast({\Delta \mathfrak{G}}[0]) \times i^\ast(X)$ is a weak equivalence of simplicial sets.  In particular, we only require that $i^\ast({\Delta \mathfrak{G}}[0]') \to i^\ast({\Delta \mathfrak{G}}[0])$ is a weak equivalence.  Consider the composition $0 \to {\Delta \mathfrak{G}}[0]' \to {\Delta \mathfrak{G}}[0]$, this is a horn inclusion  and therefore is a trivial cofibration, and in particular a weak equivalence.  Consequently, by two-out-of-three,  the map $i^\ast({\Delta \mathfrak{G}}[0]') \to i^\ast({\Delta \mathfrak{G}}[0])$ is a weak equivalence as required.  Therefore the unit axiom holds.
\end{proof}

\subsection{Examples}

\begin{example}[Lifting (1-)stacks to $SO(2)$-equivariant stacks]\leavevmode

This example will give a very brief overview general machinery for constructing $SO(2)$-equivariant stacks, taken from~\cite{balchin}.  For a groupoid $\mathcal{G}$, one can construct its \emph{cyclic nerve} $N^\mathfrak{C}\mathcal{G}$ which is the cyclic object which in dimension $n$ has diagrams of the form:
$$\xymatrix{x_0 \ar[r]^{a_1} & x_1 \ar[r]^{a_1} & \cdots \ar[r]^{a_{n}}& x_n \rlap{ ,}}$$
with the cyclic operator $\tau_{n}$ being defined as follows:
$$\xymatrixcolsep{3pc}\xymatrix{x_n \ar[r]^{(a_n \cdots a_1)^{-1}} & x_0 \ar[r]^{a_1} \ar[r] & \cdots \ar[r]^-{a_{n-1}} & x_{n-1} \rlap{ .}}$$
If we are given a (1-)stack $\mathcal{X} \colon \mathcal{C} \to \textbf{Grpd}$, we can apply the cyclic nerve component-wise to get a cyclic stack $N^\mathfrak{C}\mathcal{X} \colon \mathcal{C} \to \widehat{\Delta \mathfrak{C}}$.  In~\cite[\S 5]{balchin}, such a construction was used to construct the $SO(2)$-equivariant derived stack of local systems of an $SO(2)$-space.  Such a construction can be done for a handful of other crossed simplicial groups with the nerves defined in~\cite{surface}.
\end{example}

\begin{example}[Equivariant cohomology theories]
It is a well known fact that one can encode cohomology theories as mapping spaces in categories of $\infty$-stacks.  We can adjust this mantra to the equivariant setting.  Let $A$ be a sheaf of abelian groups. We define the $\mathfrak{G}$-equivariant cohomology of a site $(\mathcal{C}, \tau)$ to be
$$H^n_{\mathfrak{G}}(\mathcal{C};A) := \pi_0 \left( \Delta \mathfrak{G} \text{-} \mathbb{R}_\tau \underline{\textup{Hom}}(\ast,K^\mathfrak{G}(A,n)) \right),$$
where $K^\mathfrak{G}(A,n))$ is a $\Delta \mathfrak{G}$-version of the Eilenberg-Mac Lane space.  A full description of this for the cyclic case is given in~\cite{balchin2}, where a chain of Quillen equivalences is given to Borel type cohomology theory.
\end{example}

\begin{remark}
It seems that for any augmentation category $\mathbb{A}$, we can discuss the theory of $\textup{Ho}(\mathbb{A}_\text{Kan})$-cohomology theory.  The obstruction to being able to do this in generality is the construction of the correct Eilenberg-Mac Lane spaces.
\end{remark}

\section{Stable Stacks}\label{sec:7}

Our second example will be that of dendroidal sets as introduced in~\cite{dendroidal}.  Unfortunately, we do not currently have any sensible examples of stable derived stacks, but nonetheless, we prove that the theory does exist.  A comprehensive and readable overview of the theory of dendroidal sets can be found in~\cite{MR2778589}. We will begin by recalling the necessary theory.

\begin{definition}
The category $\Omega$ of trees as objects finite rooted trees, i.e., an object of $\Omega$ is of the form:

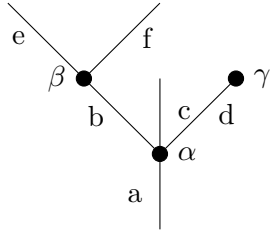
\begin{figure}[h]
\centering
{\scalefont{1.0}
\begin{tikzpicture}
\node [fill,circle,draw,inner sep = 0pt, outer sep = 0pt, minimum size=2mm, label=right:$\alpha$] (1) at (0,0) {};
\node [fill,circle,draw,inner sep = 0pt, outer sep = 0pt, minimum size=2mm, label=left:$\beta$] (2) at (-1,1) {};
\node [fill,circle,draw,inner sep = 0pt, outer sep = 0pt, minimum size=2mm, label=right:$\gamma$] (3) at (1,1) {};
\draw (1) edge node [left=1mm] {b} (2);
\draw (0,-1) edge node [left=1mm] {a}  (1);
\draw (1) edge node [right=1mm] {d} (1,1);
\draw(2) edge node [right=1mm] {f} (0,2);
\draw(2) edge node [left=1mm] {e} (-2,2);
\draw (1) edge node [right=1mm]{c} (0,1);
\end{tikzpicture}}
\end{figure}

 Any such tree $T$ generates a symmetric coloured operad $\Omega(T)$ whose edge set $E(T)$ of the edges of $T$.  The morphisms $T \to T'$ in $\Omega$ are the maps of the symmetric coloured operads $\Omega(T) \to \Omega(T')$.  The category of \emph{dendroidal sets} is the presheaf category $\widehat{\Omega}$.
\end{definition}

\begin{proposition}
The category $\Omega$ is an augmentation category.
\end{proposition}

\begin{proof}
We check that $\Omega$ satisfies the conditions (AC1) -- (AC3) appearing in Definition~\ref{augdef}:
\begin{enumerate}
\item[(AC1)] First we require that $\Omega$ is an EZ-category.  This is given in \cite[Example 7.6(c)]{genreedy}, but we will elaborate on some of the details here.  The degree function $d \colon \Omega \to \mathbb{N}$ is given by $d(T) = \#\{\text{vertices of }T\}$.  Every morphism in $\Omega$ can be decomposed as an automorphism (arising from considering different planar structures) followed by a series of face and degeneracy maps \cite[Lemma 3.1]{dendroidal}.  There is a category $\Omega_\text{planar}$ which fixes a planar representation of the trees, and is a strict EZ-category.  By the description of the morphisms in $\Omega$, we see that it is a crossed $\Omega_\text{planar}$ group, and by Lemma~\ref{crossisez} is a generalised EZ-category.  The degeneracy and face operators that appear in the morphism structure of $\Omega$ are exactly those we use in the EZ-structure.  We also require a monoidal product $\square$ which gives $\Omega$ the structure of a quasi-monoidal EZ-category.  In this case $\square = \otimes$, the \emph{Boardman-Vogt tensor product} on dendroidal sets as described in \cite[\S5]{dendroidal}.  We will not describe the construction of the tensor product here. 
\item[(AC2)] Next, we require an inclusion $i \colon \Delta \hookrightarrow \Omega$ which is compatible with the monoidal structure.  The inclusion is given by considering the object $[n]$ as the linear tree $L_n$ which has $n+1$ vertices and $n$ edges.  The compatibility of the tensor product with the monoidal structure on simplicial sets is given in \cite[Proposition 5.3]{dendroidal}.
\item [(AC3)] The proof of AC3 is highly technical, relying on the shuffle of trees.  However, the property is proved as the main result of~\cite{erratadendroidal}, and we will not concern ourself with the details here.
\end{enumerate}
\end{proof}

Again, we explain the construction of the left adjoint $i_!$.  In this case, it is a restriction functor, so the left adjoint is an extension by zero, which sends a simplicial set $X$ to the dendroidal set defined by
$$i_!(X)_T = \begin{cases}X_n &\mbox{if } T \simeq i([n]) \\ 
\emptyset & \mbox{otherwise} \end{cases}.$$

\begin{corollary}\label{dendroidalcorr}\leavevmode
\begin{itemize}
\item There is a Quillen model structure on the category of dendroidal sets, denoted $\widehat{\Omega}_\text{Kan}$, where the fibrant objects are the dendroidal Kan complexes, and the cofibrations are the normal monomorphisms.  Moreover there is a Quillen adjunction 
$$i_! \colon \widehat{\Delta}_\text{Kan} \rightleftarrows \widehat{\Omega}_\text{Kan} \colon i^\ast.$$
\item There is a local model structure on the category of dendroidal presheaves, denoted $\Omega \textbf{-Pr}_\tau(\mathcal{C})$ obtained as the left Bousfield localisation of the point-wise Kan model at the class of dendroidal hypercovers. Moreover there is a Quillen adjunction
$$i_! : \textbf{sPr}_\tau(\mathcal{C}) \rightleftarrows \Omega \textbf{-Pr}_\tau(\mathcal{C}) : i^\ast.$$
\item The category of $(\Omega,n)$-geometric derived Artin stacks, denoted $\Omega \text{-} \mathcal{G}_n^\text{sm}(\textbf{dAff})$, is given as the homotopy category of derived $(\Omega,n)$-hypergroupoids with respect to the class of trivial $(\Omega,n)$-hypergroupoids. 
\end{itemize}
\end{corollary}

\begin{remark}
As $\Omega_\text{planar}$ is also an augmentation category, it is possible to replace all symbols $\Omega$ by $\Omega_\text{planar}$ in Corollary~\ref{dendroidalcorr}.  Note that also because $\Omega_\text{planar}$ is a strict EZ-category, this model structure will be of Cisinski type.
\end{remark}

\begin{definition}
The augmented homotopical algebraic geometry theory arising from $\Omega$ will be referred to as \emph{stable}.  For example, the category $\text{Ho}(\Omega \textbf{-Pr}_\tau(\mathcal{C}) )$ will be referred to as the category of \emph{stable $\infty$-stacks}.
\end{definition}

\begin{remark}
Unfortunately, unlike the simplicial and crossed simplicial cases, the model structure $\widehat{\Omega}_\text{Kan}$ is not a closed monoidal category, which is the reason for asking that an augmentation category be only promagmoidal, and not promonoidal. This produced several errors in the existing literature, as it was assumed that the pushout-product axiom holds.  This error was spotted and corrected in the errata of~\cite{cisinski}, with details appearing in~\cite{erratadendroidal}.  However, it is possible to restrict to the category $\Omega_o$ of those trees which have no bald vertices.  In this case, the pushout-product axiom does hold as shown in~\cite[Corollary 2.5]{erratadendroidal}.  It is also true that $\Omega_o$ is an augmentation category.  However, using this category for the purpose of augmented homotopical algebraic geometry is still unreasonable as there is no known description of the homotopy type of $\Omega_o$.
\end{remark}

We finish this section by discussing the homotopy type of $\widehat{\Omega}_\text{Kan}$.  This section will lead to the choice of the quantifier $\emph{stable}$ to describe $\Omega$-augmentation.  The results in this section  are built on~\cite{basicthesis}, where the model $\widehat{\Omega}_\text{Kan}$ was explicitly described.  Note that in the aforementioned reference, the Kan model structure is called the \emph{stable model structure}.

We first highlight a difference to what one may find in the literature.  The model structure on $\widehat{\Omega}$ that one usually encounters is the \emph{operadic} model structure (sometimes referred to the \emph{Cisinski-Moerdijk} structure). The fibrant objects in this model structure are the \emph{inner Kan complexes}.  This model structure is Quillen equivalent to a certain model structure on $\textbf{sOperad}$.  It is shown in~\cite{basicthesis} that the Kan model structure is a localisation of this model structure at the collection of \emph{outer horns}.  One can easily see the duality between this construction and the relationship between $\widehat{\Delta}_\text{Kan}$ and $\widehat{\Delta}_\text{Joyal}$.

Recall from~\cite{MR1361893} that a \textit{spectrum} $\underline{E} = \{E_i\}_{i \in \mathbb{Z}}$ is a sequence of based spaces $E_n$ and based homeomorphisms $E_i \simeq \Omega E_{i+1}$.  We say that a spectrum is \textit{connective} if $E_i = 0$ for $i < 0$. We can view a connective spectrum as an infinite loop space via delooping machinery as in~\cite{MR0339152}, or equally as a $\Gamma$-space as in~\cite{MR0353298}.  We will denote by $\textbf{ConSp(Top)}$ the category of connective spectra.

\begin{proposition}[\cite{MR513569}, Proposition  2.2]
There is a model structure on the category \textbf{ConSp(Top)}, called the \textit{stable model structure}, where a map $f \colon X \to Y$ is a:
\begin{itemize}
\item Weak equivalence if $f_\ast \colon \pi_\ast X \simeq \pi_\ast Y$ where $\pi_\ast X = \displaystyle{\lim_{\to k}} \pi_{\ast +k} X_n$.
\item Fibration if it has the RLP with respect to trivial cofibrations.
\item Cofibration if $f_n \colon X_n \to Y_n$ is a cofibration of spaces for all $n \geq 0$.
\end{itemize}
\end{proposition}

\begin{proposition}[\cite{MR3349322}, Theorem 5.4]\label{conspec}
The Kan model structure on dendroidal sets is Quillen equivalent to the stable model structure on connective spectra.
\end{proposition}

\subsection{Amalgamations}
We now describe a way to combine the examples of ${\Delta \mathfrak{G}}$ and $\Omega$ into a new augmentation category using the theory of categorical pushouts. Usually, taking a categorical pushout leads to an unwieldy result, however, we will show that it is of the form of an \emph{amalgamation}.

It is known that if you take the algebraic free products with amalgamations of groups, then the original groups embed into this new group \cite{MR3069472}.  An amalgamation in an arbitrary category is a pushout along two monic maps.  

\begin{definition}
A category is said to have the \emph{amalgamation property} if amalgamations exists for all diagrams of monic maps $B \hookleftarrow A \hookrightarrow C$.  
\end{definition}

It can be shown that such diagrams share the analogous embedding property of free products with amalgamations of groups.  Although the category of small categories, $\textbf{Cat}$, does not satisfy the amalgamation property for all pushouts, there are some sufficient conditions which ensure the amalgamation property, this will be our next topic of focus.

Recall that for any functors $F_X \colon \mathcal{W} \to \mathcal{X}$ and $F_Y \colon \mathcal{W} \to \mathcal{Y}$ we can form the pushout category $\mathcal{Z}$ which makes the following universal diagram of functors commutative:
$$\xymatrix{\mathcal{W}  \ar[r]^{F_Y}  \ar[d]_{F_X}& \mathcal{Y} \ar[d]^{G_Y} \\
\mathcal{X} \ar[r]_{G_X} & \mathcal{Z} \rlap{ .}}$$
For the rest of this section we will assume that $\mathcal{W}$ is a subcategory of both $\mathcal{X}$ and $\mathcal{Y}$ and that the functors $F_X$ and $F_Y$ are embeddings.

\begin{definition}
We will say that the pushout category $\mathcal{Z}$ is an \textit{amalgamation} if the functors $G_X$ and $G_Y$ are embeddings.
\end{definition}

Not all pushouts of this form are an amalgamation, some assumptions are needed on the embeddings $F_X$ and $F_Y$.  This was first considered by Trnkov\'{a} in \cite{amalgamations2}, but a more general condition was proved by MacDonald and Scull in \cite{amalgamations}.

\begin{definition}[{\cite[Definition 3.1]{amalgamations}}]
A class of morphisms $\mathcal{M}$ of $\mathcal{X}$ has the \textit{3-for-2 property} when if $f,g$ and $h=g \circ f$ are morphisms in $\mathcal{X}$, if any two of $f,g$ and $h$ are in $\mathcal{M}$, then the third is also in $\mathcal{M}$.
\end{definition}

\begin{definition}[{\cite[Definition 3.2]{amalgamations}}]
A functor $F_Y \colon \mathcal{W} \to \mathcal{Y}$ has the \textit{3-for-2 property} if the set of image morphisms $\mathcal{F} := \{ F_Y(b) \mid b \text{ a morphism in } \mathcal{W} \}$ satisfies the 3-for-2 property.
\end{definition}

\begin{theorem}[{\cite[Theorem 3.3]{amalgamations}}]
If the functors $F_X \colon \mathcal{W} \to \mathcal{X}$ and $F_Y \colon \mathcal{W} \to \mathcal{Y}$ are embeddings which both satisfy the 3-for-2 property, then the induced functors $G_X \colon \mathcal{X} \to \mathcal{Z}$ and $G_Y \colon \mathcal{Y} \to \mathcal{Z}$ are also embeddings.  Therefore $\mathcal{Z}$ is an amalgamation.
\end{theorem}

Note that any full functor is automatically 3-for-2.  The following lemma gives the structure of the category $\mathcal{Z}$.

\begin{lemma}[{\cite[Lemma 10.2]{Fiore08modelstructures}}]\label{paolistruct}
Let $\mathcal{Z}$ be an amalgamation of $\mathcal{X} \leftarrow \mathcal{W} \rightarrow \mathcal{Y}$ such that the map $\mathcal{W} \to \mathcal{X}$ is full.  Then:
\begin{itemize}
\item $\text{Ob}(\mathcal{Z}) = \text{Ob}(\mathcal{Y}) \coprod (\text{Ob}(\mathcal{X}) \backslash \text{Ob}(\mathcal{W})).$
\item The morphisms of $\mathcal{Z}$ have two forms:
	\begin{enumerate}
	\item A morphism $\xymatrix{X_0 \ar[r]^f & X_1}$ with $f \in \text{Mor}(\mathcal{X}) \backslash \text{Mor}(\mathcal{W})$.
	\item A path $\xymatrix{X_0 \ar[r]^{f_1} & Y_1 \ar[r]^d & Y_2 \ar[r]^{f_2} & X_2}$ where $d$ is a morphism in $\mathcal{Y}$, and $f_1,f_2 \in \text{Mor}(\mathcal{X}) \backslash \text{Mor}(\mathcal{W}) \cup \{ \text{identities on } \text{Ob}(\mathcal{Z})\}$.  If $f_1$ is non-trivial then $Y_1 \in \mathcal{W}$.  If $f_2$ is non-trivial, then $Y_2 \in \mathcal{W}$.
	\end{enumerate}
\end{itemize}
\end{lemma}

We shall now discuss when an augmentation category has the amalgamation property.  We will then prove that for an amalgable augmentation category $\mathbb{A}$, the pushout ${\Delta \mathfrak{G}} \sqcup_\Delta \mathbb{A}$ is an augmentation category.  First, recall that any morphism in $\Delta$ can be decomposed in the form \textbf{SD} where \textbf{S} is a composition of degeneracy maps, and \textbf{D} is a composition of face maps.  Furthermore, any morphism in a crossed simplicial group ${\Delta \mathfrak{G}}$, by definition, can be written in the form \textbf{TSD} for \textbf{SD} as above and \textbf{T} a composition of morphisms of some $\mathfrak{G}_n$. We can then use the properties of crossed simplicial groups to rearrange this in to the form $\textbf{S}\mathbf{'}\textbf{T}\mathbf{'}\textbf{D}\mathbf{'}$. 

\begin{definition}
Let $\mathcal{S}$ be a fully faithful subcategory of $\mathcal{C}$. We say that $\mathcal{S}$ is a \emph{sieve} in $\mathcal{C}$ if for every morphism $f \colon c \to s$ in $\mathcal{C}$ with $s \in \mathcal{S}$, $c$ and $f$ are also in $\mathcal{S}$.
\end{definition}

\begin{definition}
An augmentation category $\mathbb{A}$ is \emph{amalgable} if $\Delta$ is a sieve in $\mathbb{A}$.
\end{definition}

\begin{proposition}
Let $\mathbb{A}$ be an amalgable augmentation category and ${\Delta \mathfrak{G}}$ a crossed simplicial group. Then the category $\mathbb{A} \mathfrak{G} := \Delta \mathfrak{G} \sqcup_\Delta \mathbb{A}$ is an amalgamation.
\end{proposition}

\begin{proof}
By assumption, the inclusion $\Delta \to \mathbb{A}$ is a sieve, and consequently full. Therefore we need only check that the map $j \colon \Delta \to {\Delta \mathfrak{G}}$ is 3-for-2.  As this map is \emph{not} full, we must check it explicitly.

We are interested in the image set $\mathcal{F} := \{j(b) \mid b \text{ is a morphism in } \Delta\}$.  Recall that $\mathcal{F}$ has the 3-for-2 property if when if two of $f,g,h=g \circ f$ are in $\mathcal{F}$ then the third is also in $\mathcal{F}$.  Assume that $f,g \in \mathcal{F}$ then $f=\textbf{S}_f\textbf{D}_f$, $f=\textbf{S}_g\textbf{D}_g$ and $h = g \circ f = \textbf{S}_f\textbf{D}_f \textbf{S}_g\textbf{D}_g \cong \textbf{S}_h \textbf{D}_h$ for some composition of face and degeneracy maps (we can always do the last step due to the relations between face and degeneracy maps in $\Delta$).  Therefore $h \in \mathcal{F}$.  Now without loss of generality assume that $f,h \in \mathcal{F}$, we must show that $g \in \mathcal{F}$.  This follows again from the unique decomposition property, if $g \neq \mathcal{F}$ then $g = \textbf{S}_g\textbf{T}_g\textbf{D}_g$ for some non-trivial composition $\textbf{T}_g$.  Therefore $g \circ f = \textbf{S}_f\textbf{D}_f \textbf{S}_g \textbf{T}_g \textbf{D}_g \cong \textbf{S}_h \textbf{T}_h \textbf{D}_h$ for some composition of face, automorphism and degeneracy maps.  Therefore $g \circ f = h \not\in \mathcal{F}$ which is a contradiction.
\end{proof}

Using Lemma \ref{paolistruct}, and the fact that $\Delta$ is a sieve in $\mathbb{A}$, we can explicitly describe the structure of $\mathbb{A} \mathfrak{G}$.

\begin{lemma}\label{morstruct}
The category $\mathbb{A} \mathfrak{G}$ has the following structure:
\begin{itemize}
\item $\text{Ob}(\mathbb{A} \mathfrak{G}) = \text{Ob}(\mathbb{A})$.
\item The morphisms of $\mathbb{A} \mathfrak{G}$ have the forms:
	\begin{enumerate}
	\item  A morphism $\xymatrix{S \ar[r]^f & T}$ with $f \in \text{Mor}(\mathbb{A}) \backslash \text{Mor}(\Delta)$.
	\item A path $\xymatrix{[m] \ar[r]^d & [n] \ar[r]^{f_2} & T}$ with 
		\begin{itemize} \item $d \in \text{Mor}(\Delta \mathfrak{G})$.
		\item	$f_2 \in \text{Mor}(\mathbb{A}) \backslash \text{Mor}(\Delta) \cup \{ \text{identities on } \text{Ob}(\mathbb{A} \mathfrak{G})\}$.
	\end{itemize}
	\end{enumerate}
\end{itemize}
\end{lemma}

\begin{proof}\leavevmode
\begin{itemize}
\item We had for a general pushout that $\text{Ob}(\mathcal{\mathbb{A} \mathfrak{G}}) = \text{Ob}(\Delta \mathfrak{G}) \coprod (\text{Ob}(\mathbb{A}) \backslash \text{Ob}(\Delta)).$  Using the fact $\text{Ob}(\Delta) = \text{Ob}(\Delta \mathfrak{G})$, (as $\Delta$ is wide in $\Delta \mathfrak{G}$), we get the result for the objects of $\mathbb{A} \mathfrak{G}$.  
\item The morphisms come directly from Lemma \ref{paolistruct} with the only difference being we do not have a map $\xymatrix{S \ar[r]^{f_1} & [m]}$ in the path.  This is due to the fact that $\Delta$ is a sieve in $\mathbb{A}$ so no such map exists.
\end{itemize}
\end{proof}

\begin{definition}\label{def:amalg}
Let $\mathcal{C}$ be a category and $\mathcal{D}$ a subcategory.  We will say that $\mathcal{C}$ is \emph{$\mathcal{D}$-strict} if for all objects $d \in \mathcal{D}$ we have $\text{Aut}_\mathcal{C}(d)$ trivial.
\end{definition}

\begin{corollary}\label{amalgcross}
Let $\mathbb{A}$ be an amalgable augmentation category, then the pushout $\mathbb{A} \mathfrak{G} := \Delta \mathfrak{G} \sqcup_\Delta \mathbb{A}$ is a crossed $\mathbb{A}$-group.
\end{corollary}

\begin{proof}
Using Lemma \ref{morstruct}, we have that $\mathbb{A}$ is wide in $\mathbb{A} \mathfrak{G}$.  Next, note that because the inclusion is full, $\mathbb{A}$ is $\Delta$-strict.  Therefore we have that the morphisms $\xymatrix{[m] \ar[r]^d & [n] \ar[r]^{f_2} & T}$ can be decomposed as an automorphism of $[m]$ (namely $\text{Aut}_{\Delta \mathfrak{G}}([m])$) followed by a map in $\mathbb{A}$.
\end{proof}

Finally, we show our desired result, which shows the compatibility between amalgamations and augmentation structures.

\begin{theorem}
Let $\mathbb{A}$ be an amalgable augmentation category, then the pushout $\mathbb{A} \mathfrak{G}$ is an augmentation category.
\end{theorem}

\begin{proof}\leavevmode
\begin{enumerate}
\item[(AC1)] We have shown in Corollary \ref{amalgcross} that $\mathbb{A} \mathfrak{G}$ has the structure of a crossed $\mathbb{A}$-group and is therefore an EZ-category.  The tensor structure is provided by the tensor structure of $\mathbb{A}$.
\item[(AC2)]  We have an inclusion $i \colon \Delta \hookrightarrow \mathbb{A} \mathfrak{G}$ coming from the fact that the pushout has the amalgamation property.  Moreover this inclusion is compatible with the tensor product by construction.
\item[(AC3)] The required pushout-product property holds as we have proved that it does in $\mathbb{A}$ and ${\Delta \mathfrak{G}}$.
\end{enumerate}
\end{proof}

\begin{example}
The category $\Omega$ is amalgable as $\Delta$ is exhibted as a sieve by the inclusion.  Therefore for any crossed simplicial group $\Delta \mathfrak{G}$, we can form the category $\Omega \mathfrak{G}$.  The question is then what is the homotopy type of $\widehat{\Omega \mathfrak{G}}$ is. For $\widehat{\Omega \mathfrak{C}}$ one could expect the Kan model structure to have the homotopy type of connective spectra with $SO(2)$-action.  Note that this category should have some, even if weak, relation to the category $\Xi$ which is used for the theory of higher cyclic operads~\cite{cycoperad}. 

\end{example}

\bibliographystyle{siam}
\bibliography{Bibliography}

\end{document}